\documentclass[11 pt]{article}
\usepackage{color,amsthm,amsmath,array,amssymb,amscd,amsfonts,latexsym, url,wasysym}
\usepackage{graphicx}
\usepackage[colorlinks,allcolors=blue]{hyperref}
\usepackage{xcolor}
\usepackage{lmodern}
\usepackage{tikz}
\usetikzlibrary{decorations.pathreplacing}

\usepackage{ytableau}
\usepackage{mathtools}
\usepackage{listings}
\usepackage{longtable}
\usepackage{pythonhighlight}
\usepackage{cite}
\usepackage{rotating}
\newcommand*\circled[1]{\tikz[baseline=(char.base)]{\node[shape=circle,draw,inner sep=2pt] (char) {#1};}}

\newtheorem{theo}{Theorem}[section]
\newtheorem{prop}[theo]{Proposition}

\newtheorem{lemm}[theo]{Lemma}

\newtheorem{rema}[theo]{Remark}

\newtheorem{conj}[theo]{Conjecture}

\newtheorem{example}[theo]{Example}

\title{\bf On the Shifted Littlewood-Richardson Coefficients and the Littlewood-Richardson Coefficients}
\author{Duc Khanh Nguyen}

\date{}

\newfont{\gothic}{eufb10}

\begin{document}
\maketitle
\begin{abstract} 
We give a new interpretation of the shifted Littlewood-Richardson coefficients $f_{\lambda\mu}^\nu$ ($\lambda,\mu,\nu$ are strict partitions). The coefficients $g_{\lambda\mu}$ which appear in the decomposition of Schur $Q$-function $Q_\lambda$ into the sum of Schur functions $Q_\lambda = 2^{l(\lambda)}\sum_{\mu}g_{\lambda\mu}s_\mu$ can be considered as a special case of $f_{\lambda\mu}^\nu$ (here $\lambda$ is a strict partition of length $l(\lambda)$). We also give another description for $g_{\lambda\mu}$ as the cardinal of a subset of a set that counts Littlewood-Richardson coefficients $c_{\mu^t\mu}^{\tilde{\lambda}}$. This new point of view allows us to establish connections between $g_{\lambda\mu}$ and $c_{\mu^t \mu}^{\tilde{\lambda}}$. More precisely, we prove that $g_{\lambda\mu}=g_{\lambda\mu^t}$, and $g_{\lambda\mu} \leq c_{\mu^t\mu}^{\tilde{\lambda}}$. We conjecture that $g_{\lambda\mu}^2 \leq c^{\tilde{\lambda}}_{\mu^t\mu}$ and formulate some conjectures on our combinatorial models which would imply this inequality if it is valid.    
\end{abstract}

\tableofcontents

\section{Introduction}
Let $\lambda,\mu,\nu$ be partitions. Let $l(\lambda)$ be the length of $\lambda$, and $s_\lambda$ be the Schur function associated to the partition $\lambda$. The Littlewood-Richardson coefficients $c_{\lambda \mu}^{\nu}$ appear in the expansion (see \cite{Fulton}) 
\begin{equation*}
s_\lambda s_\mu =\sum\limits_{\nu} c_{\lambda\mu}^\nu s_\nu
.\end{equation*}
If now $\lambda,\mu,\nu$ are strict partitions, let $Q_\lambda$ be the shifted Schur $Q$-function associated to the strict partition $\lambda$. The shifted Littlewood-Richardson coefficients appear in the expansion (see \cite{Stembridge})
\begin{equation*}
Q_\lambda Q_\mu =\sum\limits_{\nu} 2^{l(\lambda)+l(\mu)-l(\nu)}f_{\lambda\mu}^\nu Q_\nu.
\end{equation*} 
For any strict partition $\lambda$, and a partition $\mu$ of the same integer, the coefficients $g_{\lambda\mu}$ appear in the decomposition (see \cite{Stembridge})
\begin{equation*}
Q_\lambda = 2^{l(\lambda)}\sum\limits_{\mu}g_{\lambda\mu}s_\mu.
\end{equation*}
The coefficients $g_{\lambda\mu}$ can be considered as shifted Littlewood-Richardson coefficients by the identity (see \cite{Stembridge})
\begin{equation*}\label{gasf} 
g_{\lambda\mu} = f_{\lambda\delta}^{\mu+\delta},
\end{equation*}
where $\delta=(l,l-1,\dots,1)$ with $l=l(\mu)$.\\

There were several developments beyond the Littlewood-Richardson rule. For example,
\begin{itemize}
\item[-]Zelevinsky \cite{Zelevinsky} expressed the coefficients $c_{\lambda\mu}^\nu$ as the number of pictures between $\mu$ and $\nu/\lambda$.  
\item[-]Remmel and Whitney \cite{RemmelWhitney} described $c_{\lambda\mu}^\nu$ as the number of standard tableaux of shape $\lambda$ constructed by some rules from the reverse filling of the skew shape $\nu/\mu$. There is also a similar version by Chen, Garsia, Remmel \cite{ChenGarsiaRemmel}  but they replaced $\lambda$ with $\nu$ and $\nu/\mu$ with $\lambda*\mu$.
\item[-]White \cite{White} showed that the set of tableaux in the construction of Remmel and Whitney \cite{RemmelWhitney} can be understood from a different point of view. It arises from Robinson-Schensted insertion of reading words of column-strict tableaux of a fixed skew shape.
\end{itemize} 
There are new approaches that come from geometry: the algorithm by Liu \cite{Liu} and the rule of Vakil \cite{RaviVakil}, etc.\\

\noindent The theory and methods for shifted Littlewood-Richardson coefficients are also developed parallelly with the theory of Littlewood-Richardson coefficients. Based on the work of Worley \cite{Worley}, Sagan \cite{Sagan}, Stembridge \cite{Stembridge}, there are several versions of the shifted Littlewood-Richardson rule for $f_{\lambda\mu}^\nu$, for example, the works of Serrano \cite{Serrano}, Cho \cite{cho2013new}, Shimozono \cite{Shimozono} and so on.   The shifted Littlewood-Richardson rule given by Stembridge \cite{Stembridge} is also re-obtained using the theory of crystal bases for the quantum queer superalgebra (see \cite{Choi18}, \cite{Grantcharov}). In \cite{Choi14}, the authors established the bijections between three models for shifted Littlewood-Richardson coefficients in \cite{Grantcharov}, \cite{Stembridge}, \cite{Serrano}.\\

In this paper, instead of re-interpreting the shifted Littlewood-Richardson rule given by Stembridge \cite{Stembridge}, we use this rule to obtain a new combinatorial model for the coefficients $f_{\lambda\mu}^\nu$ and $g_{\lambda\mu}$. The advantage of the new results allows us to compute the coefficients easier and to realize the connections with Littlewood-Richardson coefficients. The motivation of our work comes from the work of Belkale, Kumar and Ressayre \cite{BelkaleKumarRessayre}. The main results in the paper \cite{BelkaleKumarRessayre} raised some first clues about relations between shifted Littlewood-Richardson coefficients with Littlewood-Richardson coefficients. Ressayre conjectures an inequality between them in \cite{Ressayre19}. We do not use the approach from geometry as in \cite{BelkaleKumarRessayre}, but we try to develop the combinatorial model of Stembridge \cite{Stembridge} to discover the bridge between coefficients. To be more precise, we describe the results as follows. \\

\noindent Our first result, Theorem \ref{newf} is a new combinatorial model for the shifted Littlewood-Richardson coefficients. This is analogous to Remmel and Whitney's work \cite{RemmelWhitney}. The combinatorial model proposed by Shimozono in \cite{Shimozono} is analogous to White's model \cite{White}, arising from Sagan's shifted insertion \cite{Sagan}. Despite the case of Littlewood-Richardson coefficients where Remmel and Whitney's construction is identified with White's construction, our construction and Shimozono's construction do not produce the same model. Since $g_{\lambda\mu}$ can be considered as a shifted Littlewood-Richardson coefficient, we obtain a new model for $g_{\lambda\mu}$ in Theorem \ref{newg}.\\

\noindent Our second result, Theorem \ref{gTtT} is also a new combinatorial interpretation of the coefficients $g_{\lambda\mu}$. More precisely, let $\tilde{\lambda}$ be the partition such that its Young diagram is the union of shifted diagram corresponding to $\lambda$ and its reflection through the main diagonal. Let $\mu^t$ be the conjugate partition of $\mu$. We prove that $g_{\lambda\mu}$ is the cardinal of a subset of a set that counts the coefficients $c_{\mu^t\mu}^{\tilde{\lambda}}$. This implies Theorem \ref{glessthanc} that \begin{equation*}
g_{\lambda\mu} \leq c^{\tilde{\lambda}}_{\mu^t\mu}.
\end{equation*}  
We conjecture a stronger inequality (see Conjecture \ref{conjectureinequal2}) 
\begin{equation*}\label{inequal2}
g_{\lambda\mu}^2 \leq c^{\tilde{\lambda}}_{\mu^t\mu}.
\end{equation*}
Using a computer program, we checked this conjecture on many examples. Based on our combinatorial model for the coefficients $g_{\lambda\mu}$, we formulate Conjecture \ref{bij} whose validity implies Conjecture \ref{conjectureinequal2}. Evidence for Conjecture \ref{bij} is that it implies easily the equality
\begin{equation*}\label{equalgg}
    g_{\lambda\mu}=g_{\lambda\mu^t}.
\end{equation*}
The equality might be well known among experts, nevertheless, we include a geometric proof in Proposition \ref{gmumut}.\\

The paper contains four sections. In the first section,  we collect some basic background about the theory of Young tableaux and related models for Littlewood-Richardson coefficients. In the second section, we present the theory of shifted tableaux, and related models, some interpretations for shifted Littlewood-Richardson coefficients. The last two sections present our two main results on the coefficients $f_{\lambda\mu}^\nu$ and $g_{\lambda\mu}$. 
 
\section{Littlewood-Richardson Coefficients}
In this section, we present Young tableaux and related models for Littlewood-Richardson coefficients.
\subsection{Young Tableaux}

A non-negative integer sequence $\lambda = (\lambda_1, \lambda_2, \dots)$ is called a {\em partition} if $\lambda_1 \geq \lambda_2 \geq \dots$. If $\lambda=(\lambda_1,\lambda_2,\dots,\lambda_l)$ with $\lambda_l >0$ and $\sum\limits_{i=1}^l \lambda_i = n$, we write $l(\lambda)=l$, $|\lambda| = n$.\\

Each partition $\lambda$ is presented by a {\em Young diagram $Y(\lambda)$} that is a collection of boxes such that:
\begin{itemize}
\item[](D1) The leftmost boxes of each row are in a column.
\item[](D2) The numbers of boxes from top row to bottom row are $\lambda_1,\lambda_2,\dots$, respectively. 
\end{itemize}

\begin{example}
$$Y((3,2))=\begin{ytableau}
\,&\,&\,\\
\,&\,
\end{ytableau}$$
\end{example}

The reflection $\sigma(Y)$ through the main diagonal of a Young diagram $Y$ is also a Young diagram. The {\em conjugate partition} $\lambda^t$ of $\lambda$ is defined by $Y(\lambda^t)=\sigma(Y(\lambda))$. \\
 
A {\em semistandard Young tableau} of shape $\lambda$ is a filling of the Young diagram $Y(\lambda)$ by the ordered alphabet $\{1<2<\dots\}$ such that:  
\begin{itemize}
\item[](Y1) The entries in each column are strictly increasing from top to bottom.
\item[](Y2) The entries in each row are weakly increasing from left to right. 
\end{itemize}

Let $\nu =(\nu_1,\nu_2,\dots)$ and $\mu=(\mu_1,\mu_2,\dots)$ be two partitions. We say that $\nu$ is bigger than $\mu$ if and only if $\nu_i \geq \mu_i$ for all $i$, and we write $\nu \geq \mu$. In this case, we define the {\em skew Young diagram} $Y(\nu/ \mu)$ as the result of removing boxes in the Young diagram $Y(\mu)$ from the Young diagram $Y(\nu)$. We write $|\nu/\mu|=|\nu|-|\mu|$. A {\em skew Young tableau} $T$ of skew shape $\nu/\mu$ is a result of filling the skew Young diagram $Y(\nu/\mu)$ by the ordered alphabet $\{1<2<\dots\}$ satisfying the rules (Y1) and (Y2).\\

The {\em word} $w(T)$ of a Young tableau $T$ is defined to be the sequence obtained by reading the rows of $T$ from left to right, starting from bottom to top. A Young tableau of skew shape $\nu/\mu$ is said to be a {\em standard skew Young tableau} if its word is a permutation of the word $12\dots |\nu/\mu|$. The {\em transpose} of a standard skew Young tableau $T$ is also a standard skew Young tableau and it is denoted by $T^t$. 

\subsection{Row-Insertion and Product Tableau}
For a Young tableau $T$ and a positive integer $x$, we recall {\em row-insertion} $x$ to $T$ in \cite{Fulton}:
\begin{itemize}
\item[1.]We start with the first row of $T$. Set INSERT $:=x$. 
\item[2.]In the row we are considering, if no entry is bigger than INSERT, then at the end of this row, we add a new box with entry INSERT.
\item[3.]If in step 2., there exists an entry that is bigger than INSERT, then find the leftmost entry in the row we are considering satisfying the condition. Put the value INSERT holding in the box of this entry and set the value of the entry we have found as the new value of INSERT.  
\item[4.] Repeat processes 2. and 3. with the new INSERT and the next row. The process will finish when a new box with some entries is added at the end of a row of $T$ or in the row below the bottom of $T$.    
\end{itemize}
The result of row-insertion $x$ to $T$ is a Young tableau, which is denoted by $T \leftarrow x$. \\

\begin{example}
$$
\begin{ytableau}
1&2&2\\
3&4
\end{ytableau} \leftarrow 1 = \begin{ytableau}
1&1&2\\
2&4\\
3
\end{ytableau}
$$
\end{example}

For two Young tableaux $T$ and $U$, the {\em product tableau} $T.U$ is defined by 
\begin{equation*}
T.U:= (\dots((T\leftarrow x_1)\leftarrow x_2)\leftarrow \dots  \leftarrow x_{n-1})\leftarrow x_n,
\end{equation*} 
where $w(U) =x_1 x_2\dots x_n$.

\begin{example} Let
$$T= \begin{ytableau}
       1&2&2\\
       3&4
\end{ytableau} \quad \text{ and } \quad U= \begin{ytableau}
       1&2&3\\
       2&4
\end{ytableau} $$
then $w(U)=24123$ and 
$$
T.U=
 \begin{ytableau}
       1&1&2&2&2&3\\
       2&4&4\\
       3
\end{ytableau}
$$
\end{example}

\subsection{Sliding and Jeu de Taquin}
For the skew Young diagram $Y(\nu/\mu)$, an {\em inner corner} of $Y(\nu/\mu)$ is a box in the Young diagram $Y(\mu)$ such that the boxes below and to the right are not in $Y(\mu)$. An {\em outside corner} is a box in the Young diagram $Y(\nu)$ such that the boxes below and to the right are not in  $Y(\nu)$.\\

Let $T$ be a skew Young tableau of skew shape $\nu/\mu$. Let $b$ be an inner corner of $\nu/\mu$. We recall {\em sliding} $b$ out of $T$ from \cite{Fulton}:
\begin{itemize}
\item[1.]Set $\begin{ytableau}
       *(red)
\end{ytableau}  := b.$  
\item[2.]Do one of the following cases:
$$
 
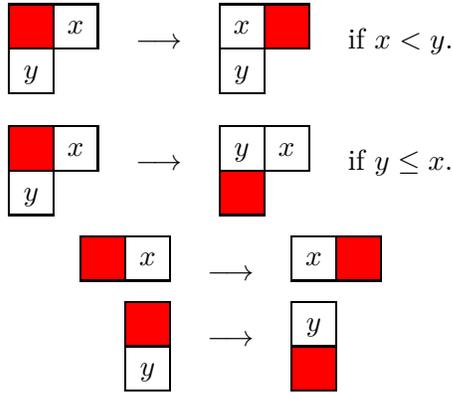
\begin{ytableau}
       *(red)&x\\
       y
\end{ytableau}\quad \longrightarrow \quad
 \begin{ytableau}
       x&*(red)\\
       y
\end{ytableau}  \quad \text{ if } x <y.
$$
$$
 \begin{ytableau}
       *(red)&x\\
       y
\end{ytableau}\quad \longrightarrow \quad
 \begin{ytableau}
       y&x\\
       *(red)
\end{ytableau} \quad \text{ if } y\leq x.
$$
$$
 \begin{ytableau}
       *(red)&x
\end{ytableau}\quad \longrightarrow \quad
 \begin{ytableau}
       x&*(red)
\end{ytableau}$$
$$
 \begin{ytableau}
       *(red)\\
       y
\end{ytableau}\quad \longrightarrow \quad
 \begin{ytableau}
       y\\
       *(red)
\end{ytableau}
$$
\item[3.]Repeat 2. until the box $\begin{ytableau}
       *(red)
\end{ytableau}$ becomes an outside corner.
\end{itemize}

The result of applying sliding $b$ out of $T$ gives us a new skew Young tableau $T'$ of skew shape $\nu'/\mu'$ such that $|\nu'| =|\nu|-1$, $|\mu'|=|\mu|-1$. Choose a random inner corner $b'$ of $T'$ and do sliding $b'$ out of $T'$ as before, we get a new skew Young tableau $T''$ of skew shape $\nu''/\mu''$ such that $|\nu''|=|\nu|-2$, $|\mu''|=|\mu|-2$. Repeating the process as many times as possible, we finally get a Young tableau and the process will terminate. There is a fact that the Young tableau we get does not depend on the choice of random inner corners in each step. The final tableau we have obtained is called the {\em rectification} of $T$ and it is denoted by $Rect(T)$. The whole process we apply on $T$ to get $Rect(T)$ is called the {\em jeu de taquin}.     

\begin{lemm}\label{wordrect}
Let $T$ and $U$ be skew Young tableaux. If $w(T)=w(U)$, then $Rect(T)=Rect(U)$.
\end{lemm}

\begin{example}
Let 
$$ T= 
 \begin{ytableau}
       *(yellow)&*(yellow)&*(yellow)&1\\
       1&2&2\\
       3&4
\end{ytableau}  
\quad \text{ and }\quad
U=
\begin{ytableau}
       *(yellow)&*(yellow)&*(yellow)&*(yellow)&1\\
       *(yellow)&1&2&2\\
       *(yellow)&4\\
       3
\end{ytableau} 
$$
The process of applying the jeu de taquin on $T$ can be visualized as follows:
$$
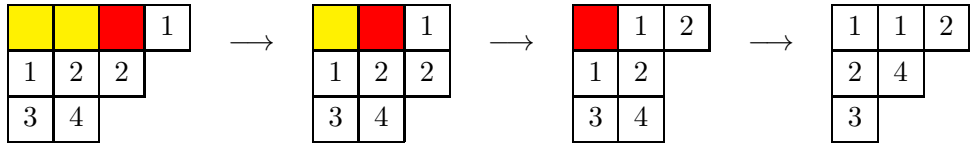
\begin{ytableau}
       *(yellow)&*(yellow)&*(red)&1\\
       1&2&2\\
       3&4
\end{ytableau} \quad 
\longrightarrow \quad
\begin{ytableau}
       *(yellow)&*(red)&1\\
       1&2&2\\
       3&4
\end{ytableau} 
\quad \longrightarrow \quad
\begin{ytableau}
       *(red)&1&2\\
       1&2\\
       3&4
\end{ytableau} 
\quad \longrightarrow \quad
\begin{ytableau}
       1&1&2\\
       2&4\\
       3
\end{ytableau} 
$$
where the boxes in red are chosen to be slid. Hence, 
$$
Rect(T)=
\begin{ytableau}
       1&1&2\\
       2&4\\
       3
\end{ytableau} 
$$
One can easily check that $Rect(U)=Rect(T)$.
\end{example}
Let $T$ and $U$ be Young tableaux. We denote by $T*U$ the new skew Young tableau which is defined as follows:
$$
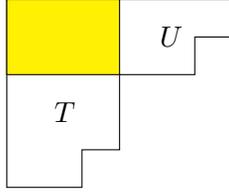
{\hbox{\begin{tikzpicture}
\draw (0,0) --++ (0,-2.5) --++ (1,0) --++ (0,0.5) --++ (0.5,0) --++ (0,1) --++ (1,0) --++ (0,0.5) --++ (0.5,0) --++ (0,0.5) -- cycle
node[below=1.5cm, right=0.5cm]{$T$};
\draw[fill=yellow]  (0,-1) --++ (1.5,0) --++ (0,1) --++ (-1.5,0) -- cycle
node[above=0.5cm, right=1.9cm]{$U$};
\end{tikzpicture}}}
$$
We have another point of view about the product tableau $T.U$ as in the following lemma (see \cite{Fulton}).
\begin{lemm}\label{prodasstar}
Let $T$ and $U$ be Young tableaux. We have $T.U =Rect(T*U)$.
\end{lemm}  

\subsection{The Robinson - Schensted - Knuth Correspondence}
A two-rowed array is defined by 
\begin{equation*}
w= \left( \begin{array}{ccc}
u_1&\dots& u_n\\
v_1& \dots& v_n
\end{array} \right),
\end{equation*} where $u_i$'s and $v_i$'s are in two independent alphabets. We say that $w$ is in {\em lexicographic order} if 
\begin{itemize}
\item[1.]$u_1 \leq u_2 \leq \dots \leq u_n$.
\item[2.] If $u_{k-1}=u_k$ for some $k$, then $v_{k-1} \leq v_k$.
\end{itemize}
The {\em Robinson - Schensted - Knuth correspondence} sets up a one-to-one correspondence between a two-rowed array in lexicographic order $w = \left( \begin{array}{ccc}
u_1 &\dots & u_n \\
v_1 &\dots & v_n
\end{array} \right)$ and a pair of tableaux of the same shape $\left( \begin{array}{c}Q\\P\end{array}  \right)$. It is described as follows (see \cite{Fulton}):\\

\underline{Create $P,Q$ from $w$:}
\begin{itemize}
\item[1.]Set $P_1 = v_1$ and $Q_1=u_1$.
\item[2.]Repeat the process below to create series
\begin{equation*}
\left( \begin{array}{c}
Q_1\\P_1
\end{array} \right), \dots, \left( \begin{array}{c}
Q_n\\P_n
\end{array}\right).
\end{equation*}
The process:
\begin{itemize}
\item[2.1.]Perform the row-insertion $v_{k+1}$ to $P_{k}$ to get $P_{k+1}$.
\item[2.2.]Let $b_{k+1}$ be the box in diagram of $P_{k+1}$ but not in the diagram of $P_{k}$. Add the box $b_{k+1}$ with entry $u_{k+1}$ into $Q_k$ to get $Q_{k+1}$. 
\end{itemize}
\item[3.] Set $Q=Q_n$ and $P=P_n$.
\end{itemize}

\underline{Create $w$ from $P,Q$:}
\begin{itemize}
\item[1.]Set $Q_n=Q$ and $P_n=P$, where $n$ is the number of entries of $P$ (or $Q$). 
\item[2.]Repeat the process below to create series
\begin{equation*}
\left( \begin{array}{c}
Q_n\\P_n
\end{array} \right), \dots, \left( \begin{array}{c}
Q_1\\P_1
\end{array}\right) \text{ and }\left( \begin{array}{c}
u_n\\v_n
\end{array}\right), \dots , \left( \begin{array}{c}
u_1\\v_1
\end{array} \right).
\end{equation*}
The process:
\begin{itemize}
\item[2.1.]Find the box farthest to the right and contain the largest entry in $Q_k$. Move out this box from $Q_k$ to get new tableau $Q_{k-1}$. Set the largest entry we have found to be $u_k$.
\item[2.2.]Perform the reverse row-insertion to $P_k$ with starting box is the box we have found. When the process finish, we get a new tableau $P_{k-1}$ and we set the entry have moved out from the first row of $P_k$ to be $v_k$.  
\end{itemize}
\item[3.] Set $w = \left( \begin{array}{ccc}
u_1 &\dots & u_n \\
v_1 &\dots & v_n
\end{array} \right)$.
\end{itemize}

We write 
\begin{equation*}
\left( \begin{array}{ccc}
u_1 &\dots & u_n \\
v_1 &\dots & v_n
\end{array} \right) \stackrel{\longleftrightarrow}{_{RSK}} \left( \begin{array}{c}Q\\P\end{array}  \right).
\end{equation*} 

\begin{example}
Let $w$ be the two-rowed array
\begin{equation*}
\left( \begin{array}{ccccccc}
1&2&3&4&5&6&7\\
2&3&6&7&4&5&1
\end{array} \right).
\end{equation*}
The series of $P_k$ and $Q_k$ are 
\begin{align*}
    P_1 &= \begin{ytableau} 2 \end{ytableau} &  Q_1 &= \begin{ytableau} 1 \end{ytableau}\\
    P_2 &= \begin{ytableau} 2&3 \end{ytableau} & Q_2 &= \begin{ytableau} 1&2 \end{ytableau}\\
    P_3 &= \begin{ytableau} 2&3&6 \end{ytableau} & Q_3 &= \begin{ytableau} 1&2&3 \end{ytableau} \\
    P_4 &= \begin{ytableau} 2&3&6&7 \end{ytableau} & Q_4 &= \begin{ytableau} 1&2&3&4 \end{ytableau} \\
    P_5 &= \begin{ytableau} 2&3&4&7\\6 \end{ytableau} & Q_5 &= \begin{ytableau} 1&2&3&4\\5 \end{ytableau} \\
    P_6 &= \begin{ytableau} 2&3&4&5\\6&7 \end{ytableau} & Q_6 &= \begin{ytableau} 1&2&3&4\\5&6 \end{ytableau} \\
    P = P_7 &= \begin{ytableau} 1&3&4&5\\2&7\\6 \end{ytableau} & Q= Q_7 &= \begin{ytableau} 1&2&3&4\\5&6\\7 \end{ytableau} \\
\end{align*}
\end{example}

\subsection{Littlewood-Richardson Rule}\label{LRsubsection}
A Young tableau $T$ is said to have {\em content} $\gamma=(\gamma_1,\gamma_2,\dots)$ if $\gamma_i$ is the number of entries $i$ in the tableau $T$. We write \begin{equation*}
x^T =x^\gamma = x_1^{\gamma_1}x_2^{\gamma_2}\dots.
\end{equation*}
For each partition $\lambda$, the {\em Schur function} $s_\lambda$ in variables $x_1,x_2,\dots$ is defined as the sum of $x^T$, where $T$ runs over the semistandard Young tableaux of shape $\lambda$. Let $\Lambda = \bigoplus\limits_{n\geq 0} \Lambda^n$ be the {\em graded ring of symmetric functions} in the variables $x_1,x_2,\dots$ with coefficients in $\mathbb{Z}$. The following set is a $\mathbb{Z}$-basis of $\Lambda^n$

$$
\left\lbrace s_\lambda \mid \lambda \text{ is a partition of }n \right\rbrace.
$$
The integers $c_{\lambda\mu}^\nu$ for each partitions $\lambda,\mu,\nu$ defined by 
\begin{equation*}
s_\lambda s_\mu =\sum\limits_{\nu} c_{\lambda\mu}^\nu s_\nu,
\end{equation*}
are called {\em Littlewood-Richardson coefficients}.\\

For any Young tableaux $V_0$ of shape $\nu$, let $\mathcal{T}(\lambda,\mu,V_0)$ be the set 
\begin{equation*}
\left\lbrace (A,U) \mid \text{$A,U$ are Young tableaux of shapes $\lambda,\mu$, respectively and $A.U =V_0$}  \right\rbrace.
\end{equation*}
For any tableau $U_0$ of shape $\mu$, let $\mathcal{S}(\nu/\lambda,U_0)$ be the set 
\begin{equation*}
\left\lbrace \text{Skew tableaux $S$ of skew shape $\nu/\lambda$ such that $Rect(S)=U_0$}  \right\rbrace.
\end{equation*}
We can describe a canonical one-to-one correspondence between $\mathcal{T}(\lambda,\mu,V_0)$ and $\mathcal{S}(\nu/\lambda,U_0)$ as follows:
\begin{itemize}
\item[1.]Let $(A,U)$ be an element of the set $\mathcal{T}(\lambda,\mu,V_0)$. Suppose that 
\begin{equation*}
\left( \begin{array}{c}
U\\
U_0
\end{array}  \right) \stackrel{\longleftrightarrow}{_{RSK}} \left(\begin{array}{ccc}
u_1&\dots&u_m\\
v_1&\dots&v_m
\end{array}  \right), 
\end{equation*}where $m=|\mu|$. Let $S$ be the new skew tableau obtained by placing $u_1,\dots,u_m$ into the new boxes while doing row-insertion $v_1,\dots,v_m$ into $A$. Then $S$ is an element of $\mathcal{S}(\nu/\lambda,U_0)$.
\item[2.]Let $S$ be an element of $\mathcal{S}(\nu/\lambda,U_0)$. Let $A'$ be an arbitrary Young tableau of shape $\lambda$. Put an order on the letters in $A'$ and $S$ in such a way that all letters in $A'$ are smaller than those in $S$. Now, suppose that 
\begin{equation*}
\left( \begin{array}{c}
V_0\\
A' \cup S
\end{array}  \right) \stackrel{\longleftrightarrow}{_{RSK}} \left(\begin{array}{ccccccc}
t_1&\dots& t_n&\quad &u_1&\dots&u_m\\
x_1&\dots&x_n&\quad&v_1&\dots&v_m
\end{array}  \right), 
\end{equation*}
where $n=|\lambda|$. Then we can construct a tableau $A$ such that $(A,U) \in \mathcal{T}(\lambda,\mu,V_0)$ by 
\begin{equation*}
\left(\begin{array}{ccc}
t_1&\dots&t_m\\
x_1&\dots&x_m
\end{array}  \right) \stackrel{\longleftrightarrow}{_{RSK}} \left( \begin{array}{c}
A\\
A'
\end{array}  \right).
\end{equation*}
\item[3.] Denote by $\mathcal{F}^{\lambda,\mu,V_0}_{\nu/\lambda,U_0}$ the map that sends $(A,U)$ in $\mathcal{T}(\lambda,\mu,V_0)$ to $S$ in $\mathcal{S}(\nu/\lambda,U_0)$. This map is a bijection between two sets (see \cite{Fulton}).
\end{itemize}

A method to compute the set $\mathcal{S}(\nu/\lambda,U_0)$ will be explained in Section \ref{symmetryLR}. In the particular case when $U_0$ is the Young tableau $\mathcal{U}_\mu$ of shape $\mu$ whose all entries in the $k$-th row are equal to $k$, one can compute explicitly the set $\mathcal{S}(\nu/\lambda,\mathcal{U}_\mu)$ by the model of Remmel and Whitney \cite{RemmelWhitney} as follows:
\begin{itemize}
\item[1.] With the skew shape $\nu/\lambda$, we number the boxes from top to bottom and right to left in each row by $1,2,\dots,|\nu/\lambda|$, respectively. The result is called the {\em reverse filling of the skew shape $\nu/\lambda$}. We denote it by $T_{\nu/\lambda}$. 
\item[2.] Define $\mathcal{O}(\nu/\lambda)$ to be the set of Young tableaux $T$ of normal shape, size $|\nu/\lambda|$, constructed from $T_{\nu/\lambda}$ satisfying the following conditions:
\begin{enumerate}
\item[](R1) If $k$ and $k+1$ appear in the same row of $T_{\nu/\lambda}$, then $k+1$ appears weakly above and strictly right of $k$ in $T$. 
\item[](R2) If $h$ appears in the box directly below $k$ in $T_{\nu/\lambda}$, then $h$ appears strictly below and weakly left of $k$ in $T$. 
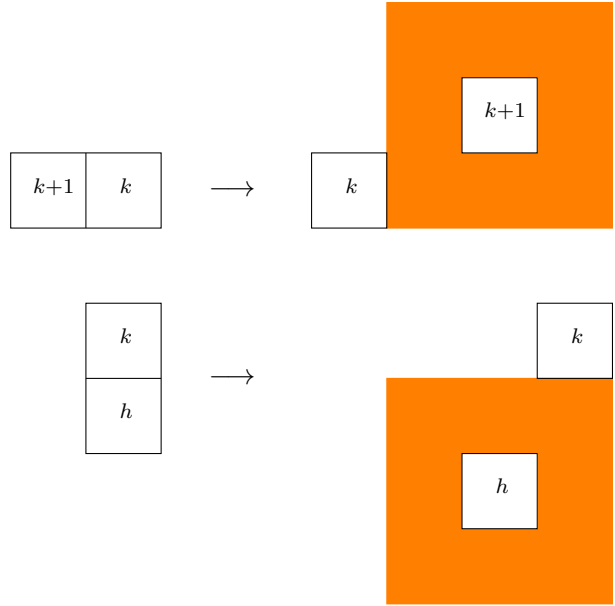
$$\hbox{\begin{tikzpicture}
\draw(0,0)--++(0,-1)--++(2,0)--++(0,1)--cycle;
\draw(1,0)--++(0,-1);
\draw(0,0)node[above=-0.5cm,right=0.1cm]{$\,^{k+1}$};
\draw(1,0)node[above=-0.5cm,right=0.25cm]{$\,^{k}$};
\draw(2.5,-0.5) node[below=0cm,right=0cm]{$\longrightarrow$};
\draw[color=orange,fill=orange](5,-1)--++(3,0)--++(0,3)--++(-3,0)--cycle;
\draw[fill=white](6,0)--++(1,0)--++(0,1)--++(-1,0)--cycle; 
\draw(6,0) node[above=0.5cm,right=0.1cm]{$\,^{k+1}$};
\draw (4,-1)--++(1,0)--++(0,1)--++(-1,0)--cycle
node[above=0.5cm,right=0.25cm]{$\,^{k}$};
\draw(1,-2)--++(1,0)--++(0,-2)--++(-1,0)--cycle;
\draw(1,-3)--++(1,0);
\draw(2.5,-3) node[above=0cm,right=0cm]{$\longrightarrow$};
\draw[color=orange,fill=orange](5,-3)--++(3,0)--++(0,-3)--++(-3,0)--cycle;
\draw(7,-2)--++(1,0)--++(0,-1)--++(-1,0)--cycle;
\draw[fill=white](6,-4)--++(1,0)--++(0,-1)--++(-1,0)--cycle;
\draw(6,-4)node[above=-0.5cm,right=0.25cm]{$\,^{h}$};
\draw(7,-2)node[above=-0.5cm,right=0.25cm]{$\,^{k}$};
\draw(1,-2)node[above=-0.5cm,right=0.25cm]{$\,^{k}$};
\draw(1,-3)node[above=-0.5cm,right=0.25cm]{$\,^{h}$};
\end{tikzpicture}}
$$ 
\end{enumerate}
\item[3.]Let $\mathcal{O}_\mu(\nu/\lambda)$ be the set of all tableaux $T$ in $\mathcal{O}(\nu/\lambda)$ of shape $\mu$. For each $T$ in $\mathcal{O}_\mu(\nu/\lambda)$, we construct a word $x_{|\mu|}\dots x_1$, where $x_k$ is the row index in $T$ containing $k$. There exists an unique skew Young tableau $T^*$ of skew shape $\nu/\lambda$ such that $w(T^*)$ is the word $x_{|\mu|} \dots x_1$. 
\item[4.]It is proved that the set $\mathcal{S}(\nu/\lambda,\mathcal{U}_\mu)$ is the set of all skew tableaux $T^*$ where $T$ runs over the set of all tableaux in $\mathcal{O}_\mu(\nu/\lambda)$.  
\end{itemize}

\begin{theo}\label{LR}Let $\lambda,\mu,\nu$ be partitions. Let $T_0$ be a Young tableau of shape $\nu$ and $U_0$ be a Young tableau of shape $\mu$. We have $c_{\lambda\mu}^\nu =\#\mathcal{O}_\mu(\nu/\lambda)=\# \mathcal{S}(\nu/\lambda,U_0)=\#\mathcal{T}(\lambda,\mu,V_0)$. 
\end{theo}
The example below shows how to compute the Littlewood-Richardson coefficients and the three models presented in Theorem \ref{LR}.
\begin{example}\label{ExampleLRrule} Set $\lambda=(3,2,1,1)$, $\mu=(4,2,1)$ and $\nu=(6,4,2,1,1)$. Then 
$$
T_{\nu/\lambda} = 
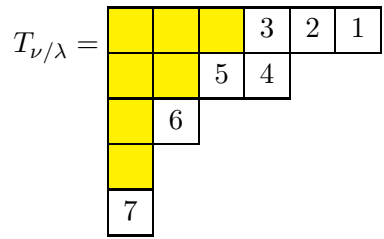
\begin{ytableau}
*(yellow)&*(yellow)&*(yellow)&3&2&1 \\
*(yellow)&*(yellow)&5&4\\
*(yellow)&6\\
*(yellow)\\
7
\end{ytableau}
$$
All tableaux of the set $\mathcal{O}_\mu(\nu/\lambda)$ are
$$
\begin{ytableau}
1&2&3&7\\
4&5\\
6
\end{ytableau} \quad
\begin{ytableau}
1&2&3&6\\
4&5\\
7
\end{ytableau} \quad
\begin{ytableau}
1&2&3&5\\
4&6\\
7
\end{ytableau} \quad
\begin{ytableau}
1&2&3&5\\
4&7\\
6
\end{ytableau} \quad
$$
We denote these tableaux by $T_1, T_2, T_3, T_4$, from left to right.
Hence, $$c_{\lambda\mu}^\nu = 4.$$ We have
$$ \mathcal{U}_\mu =
\begin{ytableau}
1&1&1&1\\
2&2\\
3
\end{ytableau}
$$
All tableaux of the set $\mathcal{S}(\nu/\lambda,\mathcal{U}_\mu)$ corresponding to $T_1, T_2, T_3, T_4$ are $$
\begin{ytableau}
*(yellow)&*(yellow)&*(yellow)&1&1&1\\
*(yellow)&*(yellow)&2&2\\
*(yellow)&3\\
*(yellow)\\
1
\end{ytableau} \quad
\begin{ytableau}
*(yellow)&*(yellow)&*(yellow)&1&1&1\\
*(yellow)&*(yellow)&2&2\\
*(yellow)&1\\
*(yellow)\\
3
\end{ytableau} \quad
\begin{ytableau}
*(yellow)&*(yellow)&*(yellow)&1&1&1\\
*(yellow)&*(yellow)&1&2\\
*(yellow)&2\\
*(yellow)\\
3
\end{ytableau} \quad
\begin{ytableau}
*(yellow)&*(yellow)&*(yellow)&1&1&1\\
*(yellow)&*(yellow)&1&2\\
*(yellow)&3\\
*(yellow)\\
2
\end{ytableau}
$$
We denote these tableaux by $S_1, S_2, S_3, S_4$, from left to right.
Set 
$$
V_0= \begin{ytableau}
1&1&2&3&4&5\\
2&6&6&7\\
3&7\\
4\\
5
\end{ytableau} \quad \text{ and }\quad
A'=
\begin{ytableau}
\color{red}{1}&\color{red}{1}&\color{red}{1}\\
\color{red}{2}&\color{red}{2}\\
\color{red}{3}\\
\color{red}{4}
\end{ytableau}
$$
with order that $\color{red}{1<2<3<4}$ $<1<2<3$. 
The two-rowed array corresponding to the pair $\left( \begin{array}{c}
V_0\\
A' \cup S_1
\end{array}  \right)$ is
\begin{equation*}
\left(\begin{array}{ccc}
\text{$\color{red}{1\,1\,1\,2\,2\,3\,4}$}&\,& 1\,1\,1\,1\,2\,2\,3 \\
1\,5\,6\,2\,4\,3\,2&\,&1\,3\,7\,7\,4\,6\,5
\end{array}  \right).
\end{equation*}
The tableaux $A_1$ and $U_1$ such that $(A_1,U_1) \in \mathcal{T}(\lambda,\mu,V_0)$ corresponding to $S_1$ are
$$A_1 = 
\begin{ytableau}
1&2&2\\
3&6\\
4\\
5
\end{ytableau}\quad \text{ and }\quad U_1= 
\begin{ytableau}
1&3&4&5\\
6&7\\
7
\end{ytableau}
$$
The tableaux $A_2$ and $U_2$ such that $(A_2,U_2) \in \mathcal{T}(\lambda,\mu,V_0)$ corresponding to $S_2$ are
$$A_2 = 
\begin{ytableau}
1&2&6\\
3&7\\
4\\
5
\end{ytableau}\quad \text{ and }\quad U_2= 
\begin{ytableau}
1&3&4&5\\
2&7\\
6
\end{ytableau}
$$
The tableaux $A_3$ and $U_3$ such that $(A_3,U_3) \in \mathcal{T}(\lambda,\mu,V_0)$ corresponding to $S_3$ are
$$A_3 = 
\begin{ytableau}
1&2&7\\
3&6\\
4\\
5
\end{ytableau}\quad \text{ and }\quad U_3= 
\begin{ytableau}
1&3&4&5\\
2&7\\
6
\end{ytableau}
$$
The tableaux $A_4$ and $U_4$ such that $(A_4,U_4) \in \mathcal{T}(\lambda,\mu,V_0)$ corresponding to $S_4$ are
$$A_4 = 
\begin{ytableau}
1&2&7\\
3&6\\
4\\
5
\end{ytableau}\quad \text{ and }\quad U_4= 
\begin{ytableau}
1&3&4&5\\
2&6\\
7
\end{ytableau}
$$

\end{example}  

\subsection{Tableau Switching}
In this section, we recall the definition and basic properties of the switching procedure in the paper \cite{Benkart}.\\

For each skew shape $\gamma$, we define a {\em perforated tableau} $T$ of shape $\gamma$ to be a result of filling some boxes in $Y(\gamma)$ with integers such that:
\begin{itemize}
\item[](PT1) The entries in each column are strictly increasing from top to bottom.
\item[](PT2) The entries in the weakly northwest of each entry $t$ of $T$ are less than or equal to $t$. 
\end{itemize}
Let $S, T$ be perforated tableaux of shape $\gamma$. We say that $S, T$ {\em fill $\gamma$} if all boxes in $Y(\gamma)$ are filled by entries of $S$ or $T$, and no box is filled twice. We then call $S\cup T$ a {\em perforated pair} of shape $\gamma$. \\

Let $S\cup T$ be a perforated pair of shape $\gamma$. Let $s$ in $S$ and $t$ in $T$ be adjacent integers, with $t$ positioned below or to the right of $s$. We define {\em switching} $s \leftrightarrow t$ by interchanging $s$ and $t$ such that after the action, both perforated tableau of shape $\gamma$ filled by entries $t$, and perforated tableau of shape $\gamma$ filled by entries $s$ satisfy the conditions (PT1) and (PT2).   \\

Choose a random pair $(s,t)$ in $S \cup T$ such that we can do the switching $s\leftrightarrow t$. Repeat this process until there are no more pairs $(s,t)$ in $S\cup T$ that can be switched $s \leftrightarrow t $. The result is a new perforated pair $T'\cup S'$ of shape $\gamma$, where $S'$ is the perforated tableau filled by entries $s$ and $T'$ is the perforated tableau filled by entries $t$. The point is that the resulting pair $T' \cup S'$ does not depend on the choices, it is denoted by $^S T \cup S_T$ (see \cite{Benkart}). The process we have done to produce $^S T \cup S_T$ from $S\cup T$ is called the {\em switching procedure}. The map that sends $S\cup T$ to $^S T \cup S_T$ is called the {\em switching map}. \\   

The example below visualizes the switching procedure.
\begin{example}Let $\gamma=(4,3,3,2)/(2,1)$. The tableau $S$ with red entries and the tableau $T$ with blue entries below are perforated tableaux of shape $\gamma$.
$$
S= \begin{ytableau}
*(yellow)& *(yellow)& \color{red}{1} &\\
*(yellow)&&\\
\color{red}{1}&&\color{red}{2}\\
&\color{red}{3}
\end{ytableau}\quad
T= \begin{ytableau}
*(yellow)&*(yellow)&&\color{blue}{-1}\\
*(yellow)&\color{blue}{-2}&\color{blue}{-2}\\
&\color{blue}{1}&\\
\color{blue}{2}&
\end{ytableau}
\quad
S\cup T = \begin{ytableau}
*(yellow)&*(yellow)&\color{red}{1}&\color{blue}{-1}\\
*(yellow)&\color{blue}{-2}&\color{blue}{-2}\\
\color{red}{1}&\color{blue}{1}&\color{red}{2}\\
\color{blue}{2}&\color{red}{3}
\end{ytableau}
$$
Look at the entries inside the circles below
$$
\begin{ytableau}
*(yellow)&*(yellow)&\color{red}{1}&\color{blue}{-1}\\
*(yellow)&\color{blue}{-2}&\color{blue}{-2}\\
\circled{\color{red}{1}}&\circled{ \color{blue}{1}}&\color{red}{2}\\
\circled{\color{blue}{2}}&\color{red}{3}
\end{ytableau}
$$
We see that we can just switch  $\color{red}{1}$ $\leftrightarrow$ $\color{blue}{1}$, but we cannot switch $\color{red}{1}$ $\leftrightarrow$ $\color{blue}{2}$. Indeed, after switching $\color{red}{1}$ $\leftrightarrow$ $\color{blue}{1}$, we get 
$$
\begin{ytableau}
*(yellow)&*(yellow)&\color{red}{1}&\color{blue}{-1}\\
*(yellow)&\color{blue}{-2}& \color{blue}{-2}\\
\color{blue}{1}&\color{red}{1}&\color{red}{2}\\
\color{blue}{2}&\color{red}{3}
\end{ytableau}
$$
The new tableau formed by the red entries and the new tableau formed by blue entries satisfies the conditions (PT1), (PT2). But after switching $\color{red}{1}$ $\leftrightarrow$ $\color{blue}{2}$, the new tableau formed by the blue entries does not satisfy the condition (PT2).\\

Here is the visualization of the switching procedure with starting point $S \cup T$ (we choose pairs in circles to switch).
$$
\begin{ytableau}
*(yellow)&*(yellow)&\color{red}{1}&\color{blue}{-1}\\
*(yellow)&\color{blue}{-2}&\color{blue}{-2}\\
\circled{\color{red}{1}}&\circled{\color{blue}{1}}&\color{red}{2}\\
\color{blue}{2}&\color{red}{3}
\end{ytableau} \quad
\longrightarrow \quad
\begin{ytableau}
*(yellow)&*(yellow)&\circled{\color{red}{1}}&\color{blue}{-1}\\
*(yellow)&\color{blue}{-2}&\circled{ \color{blue}{-2}}\\
\color{blue}{1}&\color{red}{1}&\color{red}{2}\\
\color{blue}{2}&\color{red}{3}
\end{ytableau}
\quad \longrightarrow \quad
\begin{ytableau}
*(yellow)&*(yellow)&\color{blue}{-2}&\color{blue}{-1}\\
*(yellow)&\color{blue}{-2}& \color{red}{1}\\
\color{blue}{1}&\color{red}{1}&\color{red}{2}\\
\color{blue}{2}&\color{red}{3}
\end{ytableau}
$$
Hence, $$
\,^S T = \begin{ytableau}
*(yellow)&*(yellow)&\color{blue}{-2}&\color{blue}{-1}\\
*(yellow)&\color{blue}{-2}&\\
\color{blue}{1}&&\\
\color{blue}{2}&
\end{ytableau} \quad \text{ and } \quad
S_T = 
\begin{ytableau}
*(yellow)&*(yellow)&&\\
*(yellow)&&\color{red}{1}\\
&\color{red}{1}&\color{red}{2}\\
&\color{red}{3}
\end{ytableau}
$$ 
\end{example} 

Let $S,T$ be skew tableaux. We say that $T$ extends $S$ if $T$ has skew shape $\nu/\lambda$ and $S$ has shape $\lambda/\mu$ for some partitions $\nu\geq \lambda\geq \mu$. The following theorem is a collection of some important properties in Theorem 2.2 and Theorem 3.1 in the paper \cite{Benkart}. 
\begin{theo}\label{switchingproperties}
Let $S,T$ be skew Young tableaux such that $T$ extends $S$. Then 
\begin{itemize}
\item[1.]$S_T$ and $^ST$ are skew Young tableaux, $S_T$ extends $^ST$. 
\item[2.]$^ST \cup S_T$ has the same shape as $S\cup T$.
\item[3.]$Rect(S) =Rect(S_T)$. 
\item[4.]$Rect(T)=Rect(^S T)$.
\item[5.]The switching map $S \cup T \mapsto \,^ST \cup S_T$ is an involution. 
\end{itemize}  
\end{theo}  
\begin{example}Let
$$S=\begin{ytableau}
\color{red}{1}&\color{red}{1}&\color{red}{1}\\
\color{red}{2}&\color{red}{2}\\
\color{red}{3}\\
\color{red}{4}
\end{ytableau}
\quad \text{ and }\quad
T= \begin{ytableau}
*(yellow)& *(yellow) &*(yellow) &\color{blue}{1}&\color{blue}{1}&\color{blue}{1}\\
*(yellow)&*(yellow)& \color{blue}{2}&\color{blue}{2}\\
*(yellow)&\color{blue}{3}\\
*(yellow)\\
\color{blue}{1}
\end{ytableau}
$$
Then $T$ extends $S$ and
$$\,^ST = \begin{ytableau}
\color{blue}{1}&\color{blue}{1}&\color{blue}{1}&\color{blue}{1}\\
\color{blue}{2}&\color{blue}{2}\\
\color{blue}{3}
\end{ytableau} \quad\text{ and }\quad
S_T = \begin{ytableau}
*(yellow)&*(yellow)&*(yellow)&*(yellow)&\color{red}{1}&\color{red}{1}\\
*(yellow)&*(yellow)&\color{red}{1}&\color{red}{2}\\
*(yellow)&\color{red}{2}\\
\color{red}{3}\\
\color{red}{4}
\end{ytableau}
$$
\end{example}

\subsection{The Symmetry of Littlewood-Richardson Coefficients}\label{symmetryLR}

The tableau switching provides a bijective proof of the symmetry of Littlewood-Richardson coefficients
\begin{equation*}
c_{\lambda\mu}^\nu = c_{\mu\lambda}^\nu.
\end{equation*} 
Indeed, let $A_0$ be a Young tableau of shape $\lambda$ and $U_0$ be a Young tableau of shape $\mu$. We can describe a one-to-one correspondence between $\mathcal{S}(\nu/\mu,A_0)$ and $\mathcal{S}(\nu/\lambda,U_0)$ by tableau switching as follows:
\begin{itemize}
\item[1.] Let $S$ be an element of $\mathcal{S}(\nu/\lambda,U_0)$. The switching map sends $A_0 \cup S$ to $^{A_0}S \cup (A_0)_S$. By Theorem \ref{switchingproperties}, we have $^{A_0}S = U_0$ and $Rect((A_0)_S)= A_0$. Hence, $(A_0)_S \in \mathcal{S}(\nu/\mu,A_0)$. 
\item[2.]By Theorem \ref{switchingproperties}, the switching map is an involution. Hence, the map that sends $S$ to $(A_0)_S$ is a bijection between $\mathcal{S}(\nu/\lambda,U_0)$ and $\mathcal{S}(\nu/\mu,A_0)$. We denote this map by $\mathcal{B}^{\nu/\lambda,U_0}_{\nu/\mu,A_0}$.
\end{itemize}

Let $V_0$ and $W_0$ be Young tableaux of shape $\nu$. The composition of the bijections below 
\begin{equation*}
\mathcal{T}(\lambda,\mu,V_0) \xrightarrow{{\mathcal{F}^{\lambda,\mu,V_0}_{\nu/\lambda ,U_0}}} \mathcal{S}(\nu/\lambda,U_0) \xrightarrow{{\mathcal{B}^{\nu/\lambda,U_0}_{\nu/\mu,A_0}}} \mathcal{S}(\nu/\mu,A_0) \xrightarrow{\left( \mathcal{F}^{\mu,\lambda, W_0}_{\nu/\mu ,A_0}  \right)^{-1} } \mathcal{T}(\mu,\lambda,W_0)
\end{equation*}  
gives us a bijection between the set $\mathcal{T}(\lambda,\mu,V_0)$ and the set $\mathcal{T}(\mu,\lambda,W_0)$. We denote this map by $\mathcal{S}^{\lambda,\mu,\nu}_{V_0,U_0,A_0,W_0}$.

\begin{rema}
Section \ref{LRsubsection} provides an algorithm to determine the set $\mathcal{S}(\nu/\lambda,\mathcal{U}_\mu)$. Applying then $\mathcal{B}^{\nu/\lambda,\mathcal{U}_\mu}_{\nu/\mu,A_0}$, we get an algorithm to compute  $\mathcal{S}(\nu/\mu,A_0)$ for any $A_0$.  
\end{rema}

\section{Shifted Littlewood-Richardson Coefficients}
In this section, we present the definition of Stembridge's models, geometric points of view for shifted Littlewood-Richardson coefficients in \cite{Worley}, \cite{Stembridge}.

\subsection{Shifted Tableaux}
A partition $\lambda =(\lambda_1,\lambda_2,\dots)$ is said to be {\em strict} if $\lambda_1>\lambda_2>\dots$. \\

Each strict partition $\lambda$ is presented by a {\em shifted diagram $sY(\lambda)$} that is a collection of boxes such that: 
\begin{enumerate}
\item[](SD1) The leftmost boxes of each row are in the main diagonal.
\item[](SD2) The numbers of boxes from top row to bottom row are $\lambda_1, \lambda_2,\dots$, respectively.
\end{enumerate}

A {\em shifted tableau} $T$ of shifted shape $\lambda$ is a result of filling the shifted diagram $sY(\lambda)$ by the ordered alphabet $\{1'<1<2'<2<\dots \}$ such that 
\begin{enumerate}
\item[]\label{T1}(T1) The entries in each column (from top to bottom) and in each row (from left to right) are weakly increasing.
\item[]\label{T2}(T2) The entries $k'$ in each row are strictly increasing from left to right.
\item[]\label{T3}(T3) The entries $k$ in each column are strictly increasing from top to bottom.
\end{enumerate}

The shifted tableau $T$ is said to have {\em content} $\gamma =(\gamma_1,\gamma_2,\dots)$ if $\gamma_i$ is the number of $i$ or $i'$ in $T$. We write \begin{equation*}
x^T = x^\gamma = x_1^{\gamma_1} x_2^{\gamma_2} \dots.
\end{equation*} 

Let $\nu =(\nu_1,\nu_2,\dots)$ and $\mu=(\mu_1,\mu_2,\dots)$ be two strict partitions with $\nu \geq \mu$. We define the {\em skew shifted diagram $sY(\nu/\mu)$} as the result of removing boxes in shifted diagram $sY(\mu)$ from shifted diagram $sY(\nu)$. A {\em skew shifted tableau} $T$ of skew shifted shape $\nu/\mu$ is a result of filling the shifted diagram $sY(\nu/\mu)$ by the ordered alphabet $\{1'<1<2'<2<\dots\}$ satisfying the rules (T1), (T2) and (T3). The {\em content} of a skew shifted tableau $T$ is defined by the same way as for a shifted tableau. 

\begin{example} Let $\lambda=(4,2,1)$. Then the shifted diagram $sY(\lambda)$ is  
$$
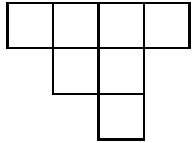
\begin{ytableau}
\,&\,&\,&\,\\
\none&\,&\,\\
\none&\none&\,
\end{ytableau}
$$
And 
$$
T=
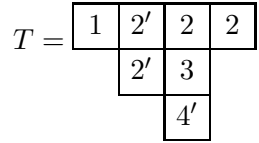
\begin{ytableau}
1&2'&2&2\\
\none&2'&3\\
\none&\none&4'
\end{ytableau}
$$
is a shifted tableau of shifted shape $(4,2,1)$. The content of $T$ is $(1,4,1,1)$.

\end{example}

\subsection{Shifted Jeu de Taquin}
For the skew shifted diagram $sY(\nu/\mu)$, we also define {\em inner corners} and {\em outside corners} by the same way as for the case of skew Young diagrams. Let $T$ be a skew shifted tableau of skew shifted shape $\nu/\mu$ without entries $k'$. Let $b$ be an inner corner of skew shifted diagram $sY(\nu/\mu)$, we define {\em shifted sliding} $b$ out of $T$, and {\em shifted jeu de taquin} on $T$, {\em shifted rectification} of $T$ which we denote by $sRect(T)$, by the same way as for the case of skew Young tableaux.\\

Here is an example of shifted jeu de taquin.
\begin{example}Set 
$$T =\begin{ytableau}
*(yellow)&*(yellow)&*(yellow)&1\\
\none &*(yellow)&2&3\\
\none &\none &4&5
\end{ytableau}$$
The process of applying the shifted jeu de taquin on $T$ can be visualized as follows:
$$
\begin{ytableau}
*(yellow)&*(yellow)&*(yellow)&1\\
\none &*(red)&2&3\\
\none &\none &4&5
\end{ytableau} \longrightarrow  
\begin{ytableau}
*(yellow)&*(yellow)&*(red)&1\\
\none &2&3&5\\
\none &\none &4
\end{ytableau} \longrightarrow 
\begin{ytableau}
*(yellow)&*(red)&1&5\\
\none &2&3\\
\none &\none &4
\end{ytableau} \longrightarrow
\begin{ytableau}
*(red)&1&3&5\\
\none &2&4
\end{ytableau} \longrightarrow
\begin{ytableau}
1&2&3&5\\
\none &4
\end{ytableau}
$$
where the boxes in red are chosen to be slid. Hence, 
$$
sRect(T)= 
\begin{ytableau}
1&2&3&5\\
\none &4
\end{ytableau}
$$
\end{example}

\subsection{Shifted Littlewood-Richardson Rule}
The {\em Schur $Q$-function} $Q_\lambda = Q_{\lambda}(x)$ in variables $x_1,x_2,\dots$ is defined as the sum of $x^T$ where $T$ runs over the shifted tableaux of shape $\lambda$. Since every coefficient in $Q_\lambda$ is divisible by $2^{l(\lambda)}$, we can define a formal power series with integer coefficients 
\begin{equation*}
P_\lambda = P_\lambda(x) := 2^{-l(\lambda)}Q_\lambda(x).
\end{equation*}
We define the {\em power-sum symmetric function} $p_r$ with $r\geq 1$ by 
\begin{equation*}
p_r = x_1^r + x_2^r + \dots .
\end{equation*}
For each partition $\lambda =(\lambda_1,\lambda_2, \dots)$, we define 
\begin{equation*}
p_\lambda = p_{\lambda_1}p_{\lambda_2}\dots.
\end{equation*}
The following set is a $\mathbb{Z}$-basis of $\Lambda^n$
$$
\left\lbrace p_\lambda \mid \lambda \text{ is a partition of }n \right\rbrace.
$$
Let $\Omega_{\mathbb{Q}} = \bigoplus\limits_{n \geq 0}\Omega_{\mathbb{Q}}^n$ be the graded subalgebra of $\Lambda_\mathbb{Q}= \mathbb{Q}\otimes_{\mathbb{Z}}\Lambda$ generated by $1,p_1,p_3,p_5,\dots$. Let $\Omega = \Omega_{\mathbb{Q}} \cap \Lambda$ be the $\mathbb{Z}$-coefficients graded subring of $\Omega_\mathbb{Q}$. We write $\Omega =\bigoplus\limits_{n\geq 0}(\Omega^n_\mathbb{Q} \cap \Lambda)$ as a graded ring.  Since 
$$
\left\lbrace P_\lambda \mid \lambda \text{ is a strict partition of }n\right\rbrace
$$ is a $\mathbb{Z}$-basis of $\Omega^n_\mathbb{Q} \cap \Lambda$, we can define integers $f_{\lambda\mu}^\nu$ for each strict partitions $\lambda,\mu,\nu$ by 
\begin{equation*}
P_\lambda P_\mu = \sum\limits_{\nu} f_{\lambda\mu}^\nu P_\nu.
\end{equation*}
The integers $f_{\lambda\mu}^\nu$ are called the {\em shifted Littlewood-Richardson coefficients}.\\

For any (skew) shifted tableau $T$, we define the {\em word} $w(T)$ to be the sequence obtained by reading the rows of $T$ from left to right, starting from bottom to top.\\

Given a word $w=w_1w_2\dots w_n$ over the alphabet $\{1'<1<2'<2<\dots\}$, we define a sequence of statistics $m_i(j)\, (0\leq j\leq 2n, i\geq 1)$ as follows:  
\begin{align*}
m_i(j)&= \text{ multiplicity of $i$ among }w_n \dots w_{n-j+1} & (0\leq j \leq n),\\
m_i(j)&= \text{ multiplicity of $i'$ among }w_1 \dots w_{j-n}& \\
&+\text{ multiplicity of $i$ among }w_n\dots w_1& (n < j\leq 2n).
\end{align*}
We say that the word $w$ is a {\em shifted lattice word} if, whenever $m_i(j)=m_{i-1}(j)$, we have
$$
w_{n-j}\ne i,i' \text{ if $0\leq j <n$},
$$
$$
w_{j-n+1} \ne i-1,i' \text{ if $n\leq j<2n$.}
$$
\begin{example}
The following words are shifted lattice words: $22'3121'111, 32212'1'111$.
\end{example}

Stembridge in \cite{Stembridge} obtained a shifted analog of the Littlewood-Richardson rule as follows.
\begin{theo}\label{StembridgeLR}
Let $\lambda,\mu,\nu$ be strict partitions. Then the coefficient $f_{\lambda\mu}^\nu$ is the number of skew shifted tableaux $T$ of skew shifted shape $\nu/\mu$ and content $\lambda$ satisfying
\begin{enumerate}
\item[](F1) The leftmost letter of $\{i,i' \text{ in }w(T)\}$ is unmarked $(1\leq i\leq l(\lambda))$.
\item[](F2) The word $w(T)$ is a shifted lattice word.
\end{enumerate}  
\end{theo}  
For each strict partition $\lambda$ and partition $\mu$ of the same integer $n$, let $g_{\lambda \mu}$ be the integer defined by 
\begin{equation*}
P_\lambda =\sum\limits_{|\mu|=n}g_{\lambda\mu}s_\mu.
\end{equation*}
In the proof of Theorem 9.3 in \cite{Stembridge}, Stembridge used the fact that 
\begin{equation}\label{glambdamu}
g_{\lambda \mu} = f_{\lambda \delta}^{\mu+\delta}, 
\end{equation}where 
\begin{equation*}
\mu=(\mu_1,\mu_2,\dots,\mu_l) \text{ a partition with }l=l(\mu),
\end{equation*}
\begin{equation*}
\delta =(l,l-1,\dots,1),
\end{equation*}
\begin{equation*}
 \mu +\delta =(\mu_1+l, \mu_2+l-1,\dots,\mu_l+1).
\end{equation*}
With the identity (\ref{glambdamu}), he obtained an explicit interpretation of $g_{\lambda\mu}$ as in the following theorem.
\begin{theo}\label{Stembridgeg}
Let $\lambda$ be a strict partition and let $\mu$be a partition. Then the coefficient $g_{\lambda\mu}$ is the number of skew shifted tableaux $T$ of shape $\mu$ and content $\lambda$ satisfying
\begin{enumerate}
\item[](G1) The leftmost letter of $\{i,i' \text{ in }w(T)\}$ is unmarked $(1\leq i\leq l(\lambda))$.
\item[](G2) The word $w(T)$ is a shifted lattice word.
\end{enumerate} 
\end{theo}

A skew shifted tableau of skew shifted shape $\nu/\mu$ is said to be {\em standard} if its word is a permutation of the word $12\dots |\nu/\mu|$. The following result can be translated equivalently from Lemma 8.4 in the paper \cite{Stembridge} of Stembridge.
\begin{theo}\label{fbystandard} Let $\lambda,\mu,\nu$ be strict partitions. Choose a standard shifted tableau $\mathcal{T}_\lambda$ of shifted shape $\lambda$. Then the coefficient $f_{\lambda\mu}^\nu$ is the number of standard skew shifted tableaux $S$ of skew shifted shape $\nu/\mu$ such that $sRect(S)=\mathcal{T}_\lambda$. 
\end{theo}

\subsection{Geometric Interpretation of the Coefficients \texorpdfstring{$f_{\lambda\mu}^{\nu}$}{TEXT} and \texorpdfstring{$g_{\lambda\mu}$}{TEXT}} 
Let $V$ be a complex vector space of dimension $m+n$. The set $Gr(m,V)$ of linear subspaces of dimension $m$ in $V$ is called a {\em complex Grassmannian}. Fix a complete flag of $V$
\begin{equation*}\label{flagF}
    \mathcal{F}: 0 =V_0 \subset \dots \subset V_i \subset \dots \subset V_{m+n}=V,
\end{equation*}
where each $V_i$ is a vector subspace of $V$ of dimension $i$. To each partition $\lambda=(\lambda_1,\dots,\lambda_m)$ with $\lambda_m \geq 0$, contained in the $m\times n$ rectangle, we associate the Schubert variety
\begin{equation*}\label{SchubertvarX}
  X_\lambda (\mathcal{F}) = \left\lbrace W \in Gr(m,V)\mid \dim(W \cap V_{n+i-\lambda_i}) \geq i\, (1 \leq i \leq m) \right\rbrace.  
\end{equation*}
The Poincare dual class of $X_\lambda(\mathcal{F})$ is denoted by $\sigma_\lambda$ and called a Schubert class. Then $\sigma_\lambda$ is an element of $H^{2|\lambda|}(Gr(m,V))$. We have (see \cite{Fulton})
\begin{equation*}
H^*(Gr(m,V)) =\bigoplus\limits_{\text{$\lambda$ is a partition contained in the $m\times n$ rectangle}} \mathbb{Z}\sigma_\lambda.
\end{equation*}

Now, let $V$ be a complex vector space $V$ of dimension $2n$, endowed with a nondegenerate skew-symmetric bilinear form $\omega$. A subspace $W$ of $V$ is isotropic if the form $\omega$ vanishes on it, i.e., $\omega(v,w)=0$ for all $v,w \in W$. A maximal isotropic subspace of $V$ is called {\em Lagrangian}. The set $LG(n,V)$ of Lagrangian subspaces in $V$ is called the {\em Lagrangian Grassmannian}. Fix a complete isotropic flag of $V$
\begin{equation*}
    \mathcal{L}: 0=V_0 \subset \dots \subset V_i \subset \dots \subset V_n\subset V,
\end{equation*}
where each $V_i$ is a vector subspace of $V$, of dimension $i$ and $V_n$ is Lagrangian. To each strict partition $\lambda =(\lambda_1,\dots,\lambda_l)$ with $\lambda_l>0$, contained in $(n,n-1,\dots,1)$, we associate the Schubert variety 
\begin{equation}\label{SchubertvarY}
   Y_\lambda(\mathcal{L}) = \left\lbrace W \in LG(n,V) \mid \dim(W \cap V_{n+1-\lambda_i}) \geq i \, (1 \leq i \leq l) \right\rbrace. 
\end{equation}
The Poincare dual class of $Y_\lambda(\mathcal{L})$ is denoted by $\theta_\lambda$ and called a Schubert class. Then $\theta_\lambda$ is an element of $H^{2|\lambda|}(LG(n,V))$. We have (see \cite{Pragacz})
\begin{equation*}
H^*(LG(n,V))=\bigoplus\limits_{\lambda \text{ is a strict partition contained in $(n,n-1,\dots,1)$}} \mathbb{Z}\theta_\lambda,
\end{equation*}
and
\begin{equation*}
\theta_\lambda \theta_\mu =\sum\limits_{\nu} 2^{l(\lambda)+l(\mu)-l(\nu)} f_{\lambda\mu}^\nu  \theta_\nu.
\end{equation*}
There is a canonical embedding $\iota: LG(n,V) \rightarrow Gr(n,V)$. The map $\iota$ induces the ring homomorphism $\iota^*:H^*(Gr(n,V)) \rightarrow H^*(LG(n,V))$. For each partition $\mu$ contained in the $n\times n$ rectangle, we have (see \cite{Pragacz2})
\begin{equation}\label{geoglambdamu}
\iota^*(\sigma_\mu) = \sum\limits_{\lambda \text{ is a strict partition contained in }(n,n-1,\dots,1)} g_{\lambda\mu} \theta_\lambda.
\end{equation} 
\subsection{Application to the Identity \texorpdfstring{$g_{\lambda\mu}=g_{\lambda\mu^t}$}{TEXT}}
\begin{prop} \label{gmumut} Let $\lambda$ be a strict partition and let $\mu$ be a partition. Then $g_{\lambda\mu} =g_{\lambda\mu^t}$.
\end{prop}

The equality has an elementary algebraic proof, which follows almost directly from the algebraic definition of the Schur $Q$-function and is probably well known. Our contribution here is a (more complicated) proof, which relies solely on the geometric interpretation of the coefficients $g_{\lambda\mu}$. The equality is a strong support for Conjecture \ref{bij} which we will see later in Remark \ref{coro_of_conj} of the Section \ref{g2<c}.

\begin{proof}(Algebraic proof)
Consider a standard involution $\omega$ on the algebra of symmetric functions which sends $p_r$ into $(-1)^{r-1}p_r$. Then (see \cite{Macdonald1998}) 
\begin{equation*}
    \omega(s_\lambda) = s_{\lambda^t}.
\end{equation*}
Since $Q_\lambda$ is generated by the power-sums with odd indices, therefore $\omega(Q_\lambda)=Q_\lambda$ and the equality $g_{\lambda\mu} = g_{\lambda\mu^t}$ follows. 
\end{proof}

\begin{proof}(Geometric proof)
Let $V$ be a complex vector space of dimension $2n$, endowed with a nondegenerate skew-symmetric bilinear form $\omega$. For each subspace $W$ of $V$, set 
\begin{align*}
    W^{\perp_\omega} &= \{v' \in V \text{ such that } \omega(v',v)=0 \text{ for all }v \in W\},\\
    W^\perp &= \{ f \in V^* \text{ such that }f(v)=0 \text{ for all }v \in W\}.
\end{align*}
Fix a complete isotropic flag of $V$
\begin{equation*}
    \mathcal{L}: 0=V_0 \subset \dots \subset V_i \subset \dots \subset V_n\subset V.
\end{equation*}
Then we can extend $\mathcal{L}$ to a complete flag $\mathcal{F}$ of $V$ as follows
\begin{equation*}
\mathcal{F}: 0=V_0 \subset \dots \subset V_i \subset \dots \subset V_{2n} =V,
\end{equation*}
where $V_{n+i} = (V_{n-i})^{\perp_\omega}$ for each $i=1,\dots,n$. Moreover, the flag $\mathcal{F}^\perp$ defined below is a complete flag of $V^*$
\begin{equation*}
\mathcal{F}^\perp: 0=(V_{2n})^\perp \subset \dots \subset (V_{2n-i})^\perp \subset \dots \subset (V_0)^\perp = V^*.
\end{equation*}
Then flag $\mathcal{L}^\perp$ defined below is a complete isotropic flag of $V^*$
\begin{equation*}
\mathcal{L}^\perp: 0=(V_{2n})^\perp \subset \dots \subset (V_{2n-i})^\perp \subset \dots \subset (V_n)^\perp \subset V^*.
\end{equation*}

We define an isomorphism $\eta: Gr(n,V) \rightarrow Gr(n,V^*)$ by $W \mapsto W^\perp$. By \cite{Harris}, we know that 
\begin{equation*}
    \eta(X_\mu(\mathcal{F}))=X_{\mu^t}(\mathcal{F}^\perp).
\end{equation*}
Hence, the map $\eta$ induces the ring isomomorphism $\eta^*:H^*(Gr(n,V^*)) \rightarrow H^*(Gr(n,V))$ with 
\begin{equation}\label{imagemu}
\eta^*(\sigma_\mu) = \sigma_{\mu^t}.
\end{equation}
The restriction of $\eta$ on $LG(n,V)$ is also an isomorphism and we still denote it by $\eta$. We have 
\begin{equation*}
    \eta(Y_\lambda(\mathcal{L})) = Y_\lambda(\mathcal{L}^\perp).
\end{equation*}
Indeed, 
\begin{itemize}
    \item[1.] For each $W \in Y_\lambda(\mathcal{L})$, we have $W^{\perp_\omega}=W$, and  
    \begin{equation*}
        \dim(W \cap V_{n+1-\lambda_i})^{\perp_\omega} = 2n-1+\lambda_i - \dim( W \cap V_{n-1+\lambda_i}).
    \end{equation*}
    Then we can rewrite (\ref{SchubertvarY}) as
    \begin{equation}\label{SchubertvarY*}
   Y_\lambda(\mathcal{L}) = \left\lbrace W \in LG(n,V) \mid \dim(W \cap V_{n-1+\lambda_i}) \geq i +\lambda_i-1 \, (1 \leq i \leq l) \right\rbrace. 
\end{equation}
    \item[2.] Now, for any $W \in Y_\lambda(\mathcal{L})$ given by (\ref{SchubertvarY*}), we have
    \begin{equation*}
        \dim(W \cap V_{n-1+\lambda_i})^{\perp} = 2n+1-\lambda_i - \dim( W^\perp \cap V_{n-1+\lambda_i}^\perp).
    \end{equation*}
    Then $W^\perp \in Y_\lambda(\mathcal{L}^\perp)$ given by (\ref{SchubertvarY}). 
\end{itemize}
Hence, the map $\eta$ induces the ring isomorphism $\eta^*:H^*(LG(n,V^*)) \rightarrow H^*(LG(n,V))$ with 
\begin{equation}\label{imagelambda}
    \eta^*(\theta_\lambda) = \theta_\lambda.
\end{equation}
We have 
\begin{equation}\label{commute}
\eta^* \iota^* =\iota^*\eta^*.
\end{equation}  
Apply $\eta^*$ on both sides of the equality (\ref{geoglambdamu}), with the help of (\ref{imagemu}),  (\ref{imagelambda}), (\ref{commute}), we get
\begin{align*}
\eta^*(\iota^*(\sigma_\mu))&= \sum\limits_{\lambda \text{  is a strict partition contained in $(n,n-1,\dots,1)$}} g_{\lambda\mu} \theta_\lambda \\
&= \sum\limits_{\lambda \text{ is a strict partition contained in $(n,n-1,\dots,1)$}}g_{\lambda\mu^t} \theta_\lambda. 
\end{align*}It implies $g_{\lambda\mu}=g_{\lambda\mu^t}$.
\end{proof}
\section{A New Combinatorial Models for the Coefficients \texorpdfstring{$f_{\lambda\mu}^\nu$}{TEXT}}
Given a skew shifted shape $\nu/\mu$, we number the boxes from top to bottom and right to left in each row by $1,2,\dots,|\nu/\mu|$, respectively. The result is called the {\em shifted reverse filling of the skew shifted shape $\nu/\mu$}. We denote it by $\widetilde{T}_{\nu/\mu}$. \\

For each $k =1,2,\dots, |\nu/\mu| $, let $k^*$ to be meant $k$ or $k'$.\\

We now let $\widetilde{\mathcal{O}}(\nu/\mu)$ be the set of all tableaux $T$ of size $|\nu/\mu|$, unshifted shape, labelled by the alphabet $\mathcal{A} = \{1'<1<2'<2<\dots <|\nu/\mu|'<|\nu/\mu|\}$, satisfying the following conditions:
\begin{enumerate}
\item[](C1) If $k$ and $k+1$ appear in the same row of $\widetilde{T}_{\nu/\mu}$, then $(k+1)^*$ appears weakly above $k$ or $(k+1)^*$ appears strictly above $k'$ in $T$. 
\item[](C2) If $h$ appears in the box directly below $k$ in $\widetilde{T}_{\nu/\mu}$, then $h^*$ appears weakly below $k'$ or $h^*$ appears strictly below $k$ in $T$. 
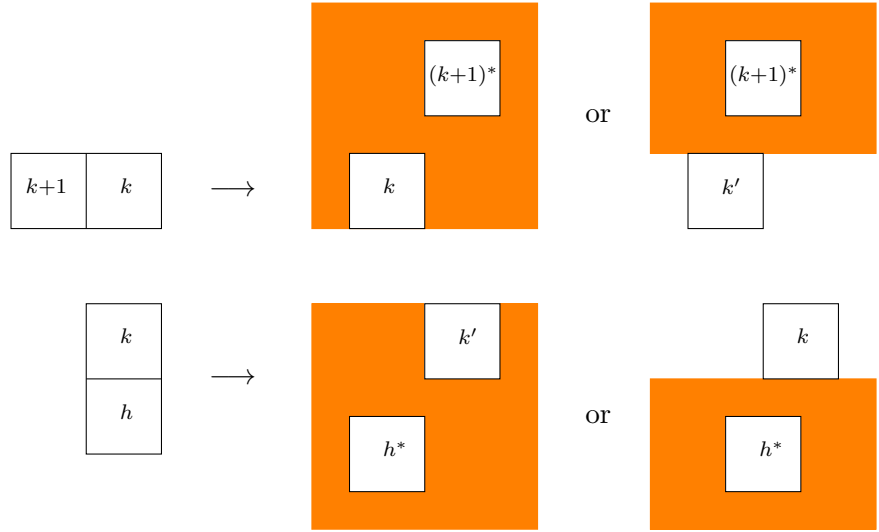
$$\hbox{\begin{tikzpicture}
\draw(0,0)--++(0,-1)--++(2,0)--++(0,1)--cycle;
\draw(1,0)--++(0,-1);
\draw(2.5,-0.5) node[below=0cm,right=0cm]{$\longrightarrow$};
\draw(1,-2)--++(1,0)--++(0,-2)--++(-1,0)--cycle;
\draw(1,-3)--++(1,0);
\draw(2.5,-3) node[above=0cm,right=0cm]{$\longrightarrow$};
\draw[color=orange,fill=orange](4,-1)--++(3,0)--++(0,3)--++(-3,0)--cycle;
\draw[color=orange,fill=orange](8.5,0)--++(3,0)--++(0,2)--++(-3,0)--cycle;
\draw[color=orange,fill=orange](4,-2)--++(3,0)--++(0,-3)--++(-3,0)--cycle;
\draw[color=orange,fill=orange](8.5,-3)--++(3,0)--++(0,-2)--++(-3,0)--cycle;
\draw[fill=white](4.5,-1)--++(1,0)--++(0,1)--++(-1,0)--cycle;
\draw[fill=white](5.5,0.5)--++(1,0)--++(0,1)--++(-1,0)--cycle;
\draw[fill=white](4.5,-4.5)--++(1,0)--++(0,1)--++(-1,0)--cycle;
\draw[fill=white](5.5,-2)--++(1,0)--++(0,-1)--++(-1,0)--cycle;
\draw[fill=white](9,-1)--++(1,0)--++(0,1)--++(-1,0)--cycle;
\draw[fill=white](9.5,0.5)--++(1,0)--++(0,1)--++(-1,0)--cycle;
\draw[fill=white](9.5,-3.5)--++(1,0)--++(0,-1)--++(-1,0)--cycle;
\draw[fill=white](10,-3)--++(1,0)--++(0,1)--++(-1,0)--cycle;
\draw(0,0)node[above=-0.5cm,right=0cm]{$\,^{k+1}$};
\draw(1,0)node[above=-0.5cm,right=0.25cm]{$\,^{k}$};
\draw(1,-2)node[above=-0.5cm,right=0.25cm]{$\,^{k}$};
\draw(1,-3)node[above=-0.5cm,right=0.25cm]{$\,^{h}$};
\draw(4.5,-1)node[above=0.5cm,right=0.25cm]{$\,^{k}$};
\draw(5.5,0.5)node[above=0.5cm,right=-0.15cm]{$\,^{(k+1)^*}$};
\draw(9,-1)node[above=0.5cm,right=0.25cm]{$\,^{k'}$};
\draw(9.5,0.5)node[above=0.5cm,right=-0.15cm]{$\,^{(k+1)^*}$};
\draw(5.5,-2)node[above=-0.5cm,right=0.25cm]{$\,^{k'}$};
\draw(4.5,-4.5)node[above=0.5cm,right=0.25cm]{$\,^{h^*}$};
\draw(10,-3)node[above=0.5cm,right=0.25cm]{$\,^{k}$};
\draw(9.5,-3.5)node[above=-0.5cm,right=0.25cm]{$\,^{h^*}$};
\draw(7.5,0.5)node[above=0cm,right=0cm]{or};
\draw(7.5,-3.5)node[above=0cm,right=0cm]{or};
\end{tikzpicture}}
$$ 
\item[](C3) The rightmost letter in each row of $T$ is unmarked.
\item[](C4) Let $\hat{T}$ be the result of reordering each row of $T$ relatively to the alphabet $\hat{\mathcal{A}}=\{ 1<2<\dots<|\nu/\mu|<|\nu/\mu|'<\dots<2'<1'\}$. Then 
\begin{itemize}
    \item[-]The entries in each column of $\hat{T}$ are increasing.
    \item[-]If the entry $j'$ belongs to the $i$-th row of $\hat{T}$ ($i>1$), and the numbers of entries less than $j$ (with respect to $\hat{\mathcal{A}}$) in the $(i-1)$-th, $i$-th row of $\hat{T}$ are $\tau_{i-1}$, $\tau_i$, respectively, then $\tau_{i-1}>\tau_{i}$.
    \item[-]If the entry $j$ belongs to the $(i-1)$-th row of $\hat{T}$ ($i>1$), and the numbers of entries less than $j'$ (with respect to $\hat{\mathcal{A}}$) in the $(i-1)$-th, $i$-th row of $\hat{T}$ are $\tau_{i-1}$, $\tau_i$, respectively, then $\tau_{i-1} > \tau_{i}$.  
\end{itemize}
\end{enumerate}
\begin{rema} $\,$ 
\begin{itemize}
    \item Since the size of $T$ is $|\nu/\mu|$, only one of $k$ or $k'$ appears in $T$ for each $k =1,2,\dots,|\nu/\mu|$. 
    \item In the condition (C1) for the set $\widetilde{\mathcal{O}}(\nu/\mu)$, $(k+1)^*$ must appear strictly right of $k^*$ in $T$. It is similar to the condition (R1) for the set $\mathcal{O}(\nu/\mu)$. However, in the condition (C2) for the set $\widetilde{\mathcal{O}}(\nu/\mu)$, it is not necessary that $h^*$ appears weakly left of $k^*$ in $T$. It is not similar as the condition (R2) for the set $\mathcal{O}(\nu/\mu)$. Indeed, for $\nu=(3,1)$, $\mu=(1)$, in $\widetilde{T}_{\nu/\mu}$, the entry $2$ is directly above the entry $3$. But $\widetilde{\mathcal{O}}(\nu/\mu)$ contains $$T=\begin{ytableau} 1&2'&3 \end{ytableau}$$ with $3$ is on the right of $2'$.
\end{itemize}
\end{rema}

\begin{example}
We illustrate how the condition (C4) works. Let $T$ be the following tableau
$$
T =\begin{ytableau}
1&2&3&4&5'&8'&10\\
6&7'&9\\
11&12
\end{ytableau}
$$
We have 
$$
\hat{T} =\begin{ytableau}
1&2&3&4&10&8'&5'\\
6&9&7'\\
11&12
\end{ytableau}
$$
We can see that the first and the second sub-condition in (C4) are satisfied. However, the third sub-condition in (C4) is not satisfied. Indeed, the entry $9$ belongs to the second row of $\hat{T}$. The number of entries less than $9'$ in the second row and the third row of $\hat{T}$ are equal two. 
\end{example}
\begin{theo}\label{newf}Let $\lambda,\mu,\nu$ be strict partitions. Then the coefficient $f_{\lambda\mu}^\nu$ is the number of the tableaux $T$ in $\widetilde{\mathcal{O}}(\nu/\mu)$ of shape $\lambda$.
\end{theo}
\begin{proof}
Let $\widetilde{\mathcal{S}}_\lambda(\nu/\mu)$ be the set of tableaux in Theorem \ref{StembridgeLR}. Let $\widetilde{\mathcal{O}}_\lambda(\nu/\mu)$ be the set of tableaux in the set $\widetilde{\mathcal{O}}(\nu/\mu)$ of shape $\lambda$. \\

Let $T \in \widetilde{\mathcal{S}}_\lambda(\nu/\mu)$ with $w(T)=w_1w_2\dots w_{|\nu/\mu|}$. We associate $T$ with a unique tableau $T'$ of unshifted shape by the rules: For each $i=|\nu/\mu|,\dots,2,1$, we have
\begin{itemize}
\item[-]If $w_i=k$, then $|\nu/\mu|+1-i$ appears in the $k$-th row of $T'$. 
\item[-]If $w_i=k'$, then $(|\nu/\mu|+1-i)'$ appears in the $k$-th row of $T'$.  
\end{itemize}
We can easily check that $T'\in \widetilde{\mathcal{O}}_\lambda(\nu/\mu)$. Indeed,
\begin{itemize}
\item[-]$T$ has content $\lambda$ if and only if $T'$ has shape $\lambda$.
\item[-]The combination of conditions (T1), (T2) on $T$ is equivalent to the condition (C1) on $T'$ as follows: Suppose that $(k+1)^*$ belongs to the $x$-th row of $T'$ and $k^*$ belongs to the $y$-th row of $T'$, then it is equivalent to say that $w_{(|\nu/\mu|+1)-(k+1)}=x^*$ and $w_{(|\nu/\mu|+1)-k}=y^*$. We see that $w_{(|\nu/\mu|+1)-(k+1)}, w_{(|\nu/\mu|+1)-k}$ are in the same row of $T$ if and only if $k+1, k$ are in the same row of $\widetilde{T}_{\nu/\mu}$. In this case, the combination of conditions (T1), (T2) on $w_{(|\nu/\mu|+1)-(k+1)}, w_{(|\nu/\mu|+1)-k}$ says that $x^* \leq y^*$ and $x^*<y'$ if $y^*=y'$ (with respect to $\mathcal{A}$). It is equivalent to the conditions $x \leq y$ and $x < y$ if $k^* = k$. Hence the combinations of (T1), (T2) on $T$ is equivalent to the condition (C1) on $T'$. 

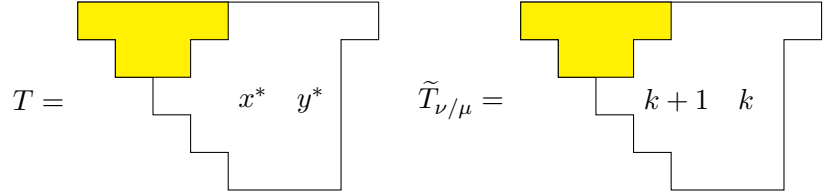
$$
{\hbox{\begin{tikzpicture}
\draw (0,0)--++(0,-0.5)--++(0.5,0)--++(0,-0.5)--++(0.5,0)--++(0,-0.5)--++(0.5,0)--++(0,-0.5)--++(0.5,0)--++(0,-0.5)--++(1.5,0)--++(0,2)--++(0.5,0)--++(0,0.5)--cycle node[below=1.3cm, right=2cm]{$x^* \quad y^*$} node[below=1.3cm, right=-1cm]{$T=$}  ;
\draw[fill=yellow](2,0)--++(0,-0.5)--++(-0.5,0)--++(0,-0.5)--++(-1,0)--++(0,0.5)--++(-0.5,0)--++(0,0.5)--cycle;
\end{tikzpicture}}}
\quad 
{\hbox{\begin{tikzpicture}
\draw (0,0)--++(0,-0.5)--++(0.5,0)--++(0,-0.5)--++(0.5,0)--++(0,-0.5)--++(0.5,0)--++(0,-0.5)--++(0.5,0)--++(0,-0.5)--++(1.5,0)--++(0,2)--++(0.5,0)--++(0,0.5)--cycle node[below=1.3cm, right=1.5cm]{$k+1 \quad k$} node[below=1.3cm, right=-1.5cm]{$\widetilde{T}_{\nu/\mu}=$};
\draw[fill=yellow](2,0)--++(0,-0.5)--++(-0.5,0)--++(0,-0.5)--++(-1,0)--++(0,0.5)--++(-0.5,0)--++(0,0.5)--cycle;
\end{tikzpicture}}}
$$

\item[-] The combination of conditions (T1), (T3) on $T$ is equivalent to the condition (C2) on $T'$ as follows: Suppose that $k^*$ belongs to the $x$-th row of $T'$ and $h^*$ belongs to the $y$-th row of $T'$, then it is equivalent to say that $w_{(|\nu/\mu|+1)-k}=x^*$ and $w_{(|\nu/\mu|+1)-h}=y^*$. We see that $w_{(|\nu/\mu|+1)-k}$ is directly above $w_{(|\nu/\mu|+1)-h}$ in $T$ if and only if $k$ is directly above $h$ in $\widetilde{T}_{\nu/\mu}$. In this case, the combination of conditions (T1), (T3) on $w_{(|\nu/\mu|+1)-k}, w_{(|\nu/\mu|+1)-h}$ says that $x^* \leq y^*$ and $x<y^*$ if $x^* = x$ (with respect to $\mathcal{A}$). It is equivalent to the condition $x \leq y$ and $x < y$ if $k^* = k$. Hence the combinations of (T1), (T3) on $T$ is equivalent to the condition (C2) on $T'$.

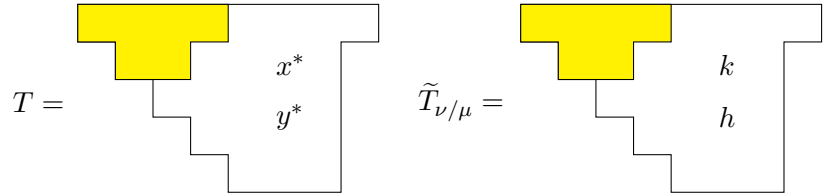
$$
{\hbox{\begin{tikzpicture}
\draw (0,0)--++(0,-0.5)--++(0.5,0)--++(0,-0.5)--++(0.5,0)--++(0,-0.5)--++(0.5,0)--++(0,-0.5)--++(0.5,0)--++(0,-0.5)--++(1.5,0)--++(0,2)--++(0.5,0)--++(0,0.5)--cycle node[below=0.8cm, right=2.5cm]{$x^*$} node[below=1.5cm, right=2.5cm]{$y^*$} node[below=1.3cm, right=-1cm]{$T=$};
\draw[fill=yellow](2,0)--++(0,-0.5)--++(-0.5,0)--++(0,-0.5)--++(-1,0)--++(0,0.5)--++(-0.5,0)--++(0,0.5)--cycle;
\end{tikzpicture}}}
\quad 
{\hbox{\begin{tikzpicture}
\draw (0,0)--++(0,-0.5)--++(0.5,0)--++(0,-0.5)--++(0.5,0)--++(0,-0.5)--++(0.5,0)--++(0,-0.5)--++(0.5,0)--++(0,-0.5)--++(1.5,0)--++(0,2)--++(0.5,0)--++(0,0.5)--cycle node[below=0.8cm, right=2.5cm]{$k$} node[below=1.5cm, right=2.5cm]{$h$} node[below=1.3cm, right=-1.5cm]{$\widetilde{T}_{\nu/\mu}=$};
\draw[fill=yellow](2,0)--++(0,-0.5)--++(-0.5,0)--++(0,-0.5)--++(-1,0)--++(0,0.5)--++(-0.5,0)--++(0,0.5)--cycle;
\end{tikzpicture}}}
$$

\item[-]The condition (F1) of $T$ is equivalent to the condition (C3) of $T'$.
\item[-]The condition (F2) of $T$ is equivalent to the conditions (C4) of $T'$. Indeed, we first translate the condition $w(T)$ is a shifted lattice word to equivalence conditions on $T'$:
\begin{itemize}
    \item[a1.] For each $j=1,\dots, |\nu/\mu| -1$, suppose that the shape of $T'$ after deleting entries $k'$ and $k>j$ is $\tau=(\tau_1,\tau_2,\dots)$, then $m_i(j) = \tau_i$ and $m_{i-1}(j) = \tau_{i-1}$. Since $w_{(|\nu/\mu|+1)-j} = x^*$ if and only if $j^*$ belongs to the $x$-th row of $T'$, the condition $$m_i(j)=m_{i-1}(j) \Rightarrow w_{|\nu/\mu|-j} \ne i,i'$$ is equivalent to  the conditions that $\tau$ is a partition and if $\tau_{i-1}=\tau_i$ for some $i>1$, then $(j+1)'$ does not belong to the $i$-th row of $T'$. 
    \item[b1.] For each $j=|\nu/\mu|,\dots,2|\nu/\mu|-1$, suppose that the shape of $T'$ after deleting entries $k'$ for $k\leq 2|\nu/\mu|-j$ is $\tau=(\tau_1,\tau_2,\dots)$, then $m_i(j) = \tau_i$ and $m_{i-1}(j) = \tau_{i-1}$. Since $w_{(|\nu/\mu|+1)-j} = x^*$ if and only if $j^*$ belongs to the $x$-th row of $T'$, then the condition
\begin{equation*}
    m_i(j)=m_{i-1}(j) \Rightarrow w_{j-|\nu/\mu|+1} \ne i-1,i'
\end{equation*}
is equivalent the conditions that $\tau$ must be a partition and if $\tau_{i-1}=\tau_{i}$ for some $i>1$, then $2|\nu/\mu|-j$ does not belong to the $(i-1)$-th row of $T'$. 
\end{itemize}
Now, we need to translate the conditions above on $T'$ to the conditions on $\hat{T'}$:
\begin{itemize}
    \item[a2.] The conditions $\tau$ is a partition in both cases is rewritten shortly by $\hat{T'}$ having strictly increasing entries (with respect to $\hat{\mathcal{A}}$) in rows and columns. Since the entries in each row of $\hat{T'}$ are strictly increasing, we just need the condition on columns. We have obtained the first sub-condition of (C4).
    \item[b2.] For each $j = 1, \dots, |\nu/\mu|-1$, we can restate the condition in a1: ``if $\tau_{i-1} = \tau_i$ for some $i>1$ then $(j+1)'$ does not belong to the $i$-th row of $T'$" to ``if $(j+1)'$ belongs to the $i$-th row of $T'$ for some $i>1$, then $\tau_{i-1}>\tau_i$". When we replace $(j+1)'$ by $j'$, then $\tau_{i-1}, \tau_i$ are the numbers of entries less than $j$ (with respect to $\hat{\mathcal{A}}$) in the $(i-1)$-th, $i$-th row of $\hat{T'}$, respectively. We have obtained the second sub-condition of (C4).
    \item[c2.] For each $j = |\nu/\mu|, \dots, 2|\nu/\mu|-1$, we can state the condition in b1: ``if $\tau_{i-1}=\tau_{i}$ for some $i>1$, then $2|\nu/\mu|-j$ does not belong to the $(i-1)$-th row of $T'$" to ``if $2|\nu/\mu|-j$ belongs to the $(i-1)$-th row of $T'$ for some $i>1$, then $\tau_{i-1}>\tau_i$". When we replace $2|\nu/\mu|-j$ by $j$, then $\tau_{i-1}, \tau_i$ are the numbers of entries less than $j'$ (with respect to $\hat{\mathcal{A}}$) in the $(i-1)$-th, $i$-th row of $\hat{T'}$, respectively. We have obtained the last sub-condition of (C4). 
\end{itemize}
    
\end{itemize}Hence, we can define an injection $\phi:\widetilde{\mathcal{S}}_\lambda(\nu/\mu)\rightarrow \widetilde{\mathcal{O}}_\lambda(\nu/\mu)$, $T\mapsto T'$. \\

Moreover, for each $T' \in \widetilde{\mathcal{O}}_\lambda(\nu/\mu)$, we associate $T'$ with a unique tableau $T$ of skew shifted shape $\nu/\mu$ and word $w(T)=w_1w_2\dots w_{|\nu/\mu|}$ by the rule:  for each $j=|\nu/\mu|,\dots,2,1$, we have
\begin{itemize}
\item[-]If $j$ appears in the $k$-th row of $T'$, then $w_{|\nu/\mu|+1-j}=k$.
\item[-]If $j'$ appears in the $k$-th row of $T'$, then $w_{|\nu/\mu|+1-j}=k'$.
\end{itemize}
The equivalence of the conditions we have already shown implies that $T\in \widetilde{\mathcal{S}}_\lambda(\nu/\mu)$. So we can define an injection $\psi:\widetilde{\mathcal{O}}_\lambda(\nu/\mu) \rightarrow \widetilde{\mathcal{S}}_\lambda(\nu/\mu), T'\mapsto T$. Moreover, $\phi\psi = Id$. Hence, $\phi$ is a bijection and $f_{\lambda\mu}^\nu =\#\widetilde{\mathcal{S}}_\lambda(\nu/\mu) =\#\widetilde{\mathcal{O}}_\lambda(\nu/\mu)$.
\end{proof}
\begin{theo}\label{newg}Let $\lambda$ be a strict partition and let $\mu$ be a partition. Then the coefficient $g_{\lambda\mu}$ is the number of the tableaux $T$ in $\widetilde{\mathcal{O}}(\mu+\delta/\delta)$ of shape $\lambda$. 
\end{theo}
\begin{proof}
This follows from Theorem \ref{newf} and identity (\ref{glambdamu}).
\end{proof}

We illustrate the method to compute the coefficients $f_{\lambda\mu}^\nu$ through an example.
\begin{example}\label{example_f}
Set $\lambda=(3,2),\mu=(3,2),\nu=(5,3,2).$\\
(1) The shifted reverse filling of the skew shifted shape $\nu/\mu$ is 
$$
\widetilde{T}_{\nu/\mu} = \begin{ytableau}
*(yellow)&*(yellow)&*(yellow)&2&1\\
\none &*(yellow)&*(yellow)&3\\
\none & \none & 5&4
\end{ytableau}
$$
(2) To construct the tableaux $T'$ in $\widetilde{\mathcal{O}}_{\lambda}(\nu/\mu)$, we first use three conditions (C1), (C2) and (C3). Then we check the results if they satisfy the condition (C4) or not.
\begin{itemize}
\item[1.]We start with $1^*$, there are two possibilities, they are $1'$ and $1$. But if $1'$ appears in the tableau $T'$, then the next position of $2^*$ will be in the row above the first row by the condition (C1). It is impossible. Hence, just only one case that $1$ appears in $T'$. Then the next two possibilities by the condition (C1) are 
$$
\begin{ytableau}
1&2
\end{ytableau}
\quad \quad
\begin{ytableau}
1&2'
\end{ytableau}
$$
\item[2.] For the second case, by the condition (C2), there are four possibilities below
$$
\begin{ytableau}
1&2'\\
3'
\end{ytableau}
\quad\quad
\begin{ytableau}
1&2'\\
3
\end{ytableau}
\quad\quad
\begin{ytableau}
1&2'&3
\end{ytableau}
\quad\quad
\begin{ytableau}
1&2'&3'
\end{ytableau}
$$
\begin{itemize}
\item[-]The last one cannot happen since the tableau $T'$ has shape $\lambda=(3,2)$. Then we consider $3'$ as the rightmost letter in the first row of $T'$ and it should be $3$ to satisfy the condition (C3).
\item[-]The second one also cannot happen because the next position of $4^*$ will be in the row below the row of $3$ by the condition (C2). It cannot produce a tableau of shape $\lambda=(3,2)$ later. 
\item[-]For the third one, the next position of $4^*$ is based on the condition (C2). To produce the shape $\lambda =(3,2)$ later, it will be as follows:
$$
\begin{ytableau}
1&2'&3\\
4
\end{ytableau}\quad\quad
\begin{ytableau}
1&2'&3\\
4'
\end{ytableau}
$$
\item[-]For the first one, the next position of $4^*$ is based on the conditions (C2) and (C3). To produce the shape $\lambda =(3,2)$ later, will be as follows:
$$
\begin{ytableau}
1&2'\\
3'&4
\end{ytableau}
$$
\end{itemize}
\end{itemize}
Continue until the end on the remaining cases by similar arguments, we finally can find the tableaux of shape $\lambda=(3,2)$ satisfying all conditions (C1), (C2) and (C3) as follows: 
$$
\begin{ytableau}
1&2&5\\
3'&4
\end{ytableau}
\quad\quad
\begin{ytableau}
1&2'&5\\
3'&4
\end{ytableau}
\quad\quad
\begin{ytableau}
1&2'&3\\
4&5
\end{ytableau}
$$
We can check that only first two tableaux above satisfy the condition (C4). Hence, 
$$f_{\lambda\mu}^\nu = 2.$$We automatically find out the set of skew shifted tableaux described in Theorem \ref{StembridgeLR} by using the bijection we mentioned in the proof of Theorem \ref{newf}. Here they are
$$
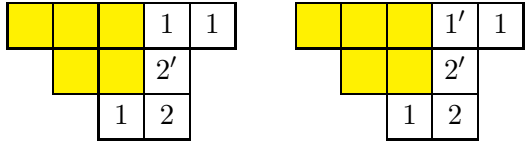
\begin{ytableau}
*(yellow)&*(yellow)&*(yellow)&1&1\\
\none &*(yellow)&*(yellow)&2'\\
\none &\none &1&2
\end{ytableau} \quad \quad
\begin{ytableau}
*(yellow)&*(yellow)&*(yellow)&1'&1\\
\none &*(yellow)&*(yellow)&2'\\
\none &\none &1&2
\end{ytableau}
$$ 
(3) The set $\widetilde{\mathcal{O}}(\nu/\mu)$ is constructed by the tree in the Appendix as follows:
\begin{itemize}
    \item[1.] We start with one node $1$.
    \item[2.] From one node $T$ with entries $1^*, \dots, k^*$, we set its child to be all possibilities of adding the next entry $(k+1)^*$ to $T$, which satisfies the conditions (C1), (C2). The process is repeated on new nodes until we can not create a new child.
    \item[3.] Finally, we eliminate all branches that end up with a node which does not satisfy the conditions (C3), (C4). The set of all ending nodes of the tree is the set $\widetilde{\mathcal{O}}(\nu/\mu)$.
\end{itemize}

\end{example}
\begin{rema}
The model of Remmel and Whitney for Littlewood-Richardson coefficients \cite{RemmelWhitney}, i.e., the set $\mathcal{O}_\mu(\nu/\lambda)$ in Theorem \ref{LR}, has another interpretation by White \cite{White}. Namely, each tableau in the set $\mathcal{O}_\mu(\nu/\lambda)$ can be considered as the recording tableau of the word rewritten in inverse order of a tableau in the set $\mathcal{S}(\nu/\lambda,\mathcal{U}_\mu)$. Our model, i.e., the set $\widetilde{\mathcal{O}}_\lambda(\nu/\mu)$, is analogous to Remmel and Whitney's model but for shifted Littlewood-Richardson coefficients. In \cite{Shimozono}, Shimozono gave an analogous model to the White's model in \cite{White} but for shifted Littlewood-Richardson coefficients. Our model and Shimozono's model are totally different. For example, with $\lambda=(3,2)$, $\mu=(3,2)$ and $\nu=(5,3,2)$, our model $\widetilde{\mathcal{O}}_\lambda(\nu/\mu)$ consists of the elements below
$$
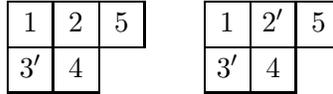
\begin{ytableau}
1 &2&5\\
3'&4
\end{ytableau}
\quad\quad
\begin{ytableau}
1&2'&5\\
3'&4
\end{ytableau}
$$
However, Shimozono's model consists of the elements below
$$
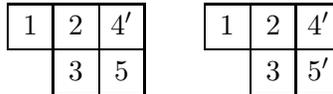
\begin{ytableau}
1&2&4'\\
\none&3&5
\end{ytableau}
\quad \quad
\begin{ytableau}
1&2&4'\\
\none&3&5'
\end{ytableau}
$$
In conclusion, the Shimozono's model is not the same as our model.
\end{rema}

\section{On the Coefficients \texorpdfstring{$g_{\lambda\mu}$}{TEXT}}\label{glambdamusection}

In this section, we present our second result. Namely, we present a new interpretation of the coefficients $g_{\lambda\mu}$ as a subset of a set that counts Littlewood-Richarson coefficients. As corollaries, we can compute the coefficients $g_{\lambda\mu}$ by models for Littlewood-Richardson coefficients. We will prove and conjecture inequalities between the coefficients and also state some conjectures that explain the hidden structure behind them.
\subsection{A New Interpretation of the Coefficients \texorpdfstring{$g_{\lambda\mu}$}{TEXT}}
For any (skew) shifted tableau $T$ without entries $k'$, let $s(T)$ be the new (skew) tableau which is defined as follows:
\begin{itemize}
\item[1.]Create an image of $T$ by the symmetry through its main diagonal.
\item[2.]Combine the image we have created with $T$ by gluing them along the main diagonal as the image below. 
$$
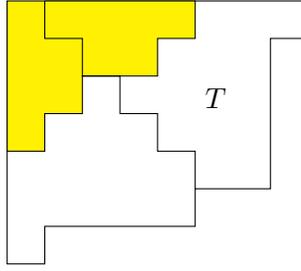
{\hbox{\begin{tikzpicture}
\draw (0,0)--++(0,-3.5)--++(0.5,0)--++(0,0.5)--++(2,0)--++(0,0.5)--++(1,0) --++(0,2)--++(0.5,0)--++(0,0.5)--++(-3.5,0) node[below=1.3cm, right=2cm]{$T$};
\draw[fill=yellow](0,-2)--++(0.5,0)--++(0,0.5)--++(0.5,0)--++(0,0.5)--++(1,0)--++(0,0.5)--++(0.5,0)--++(0,0.5)--++(-2.5,0)--cycle; 
\draw (0.5,0) --++ (0,-0.5)--++ (0.5,0) --++ (0,-0.5)--++(0.5,0) --++ (0,-0.5)--++(0.5,0) --++(0,-0.5)--++(0.5,0)--++(0,-0.5);  
\end{tikzpicture}}}
$$ 
\end{itemize}
Let $\nu/\mu$ be the skew shifted shape of $T$. We denote by $\widetilde{\nu/\mu}$ the shape of $s(T)$.
\begin{example}
$$\text{If }T= \begin{ytableau}
*(yellow)&*(yellow)&*(yellow)&1&4\\
\none &*(yellow)&*(yellow)&3\\
\none &\none &2&5
\end{ytableau} \text{ then }
s(T) =\begin{ytableau}
*(yellow)&*(yellow)&*(yellow)&*(yellow)&1&4\\
*(yellow)&*(yellow)&*(yellow)&*(yellow)&3\\
*(yellow)&*(yellow)&2&2&5\\
1&3&5\\
4
\end{ytableau}
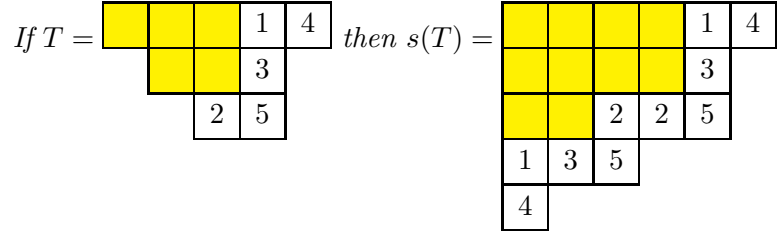
$$
\end{example}

The following result is a restatement of Proposition 5.4 in the paper \cite{Haiman}.

\begin{prop}\label{coefffrect} Let $T$ be a skew shifted tableau without entries $k'$. Then we have $s(sRect(T))=Rect(s(T))$. 
\end{prop}

For any strict partition $\lambda$ of $n$, let $\mathcal{T}_\lambda$ be the shifted tableau of shifted shape $\lambda$, obtained by putting numbers $1, 2, \dots, n$ in the boxes of shifted diagram $sY(\lambda)$ from left to right, starting from top to bottom. We recall that $\mathcal{T}(\mu^t,\mu,s(\mathcal{T}_\lambda))$ is the set of all pairs of Young tableaux $(T,U)$ of shape $\mu^t, \mu$, respectively and $T.U=s(\mathcal{T}_\lambda)$. Let $\overline{\mathcal{T}(\mu^t,\mu,s(\mathcal{T}_\lambda))}$ be the subset of $\mathcal{T}(\mu^t,\mu,s(\mathcal{T}_\lambda))$ of all pairs $(T,U)$ such that $T=U^t$. 
\begin{theo}\label{gTtT}
Let $\lambda$ be a strict partition and let $\mu$be a partition. Then 
$g_{\lambda\mu} = \# \overline{\mathcal{T}(\mu^t,\mu,s(\mathcal{T}_\lambda))}$. 
\end{theo}
\begin{proof}
We recall the identity (\ref{glambdamu}) that $g_{\lambda\mu}=f_{\lambda\delta}^{\mu+\delta}$. By Theorem \ref{fbystandard}, it is the number of standard skew shifted tableaux $S$ of skew shifted shape $(\mu+\delta)/\delta$ such that $sRect(S)=\mathcal{T}_\lambda$. The condition $sRect(S)=\mathcal{T}_\lambda$, by Proposition \ref{coefffrect} is equivalent to the condition 
\begin{equation*}
s(\mathcal{T}_\lambda) = Rect(s(S)).
\end{equation*} The tableau $s(S)$ has form 
$$
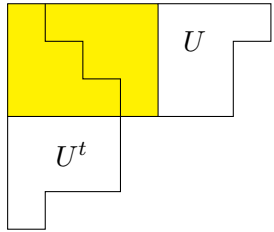
{\hbox{\begin{tikzpicture}
\draw (0,0) node[below=0.5cm, right=2.2cm]{$U$};
\draw (0,0) node[below=2cm, right=0.5cm]{$U^t$};
\draw[fill=yellow](0,0)--++(0,-1.5)--++(2,0)--++(0,1.5)--cycle; 
\draw (0.5,0)--++(0,-0.5)--++(0.5,0)--++(0,-0.5)--++(0.5,0)--++(0,-1.5)--++(-1,0)--++(0,-0.5)--++(-0.5,0)--++(0,1.5);
\draw (2,0)--++(1.5,0)--++(0,-0.5)--++(-0.5,0)--++(0,-1)--++(-1,0);
\end{tikzpicture}}}
$$ 
where $U$ is a standard Young tableau of shape $\mu$. Since $s(S)$ and $U^t*U$ have the same word, then by Lemma \ref{wordrect} and Lemma \ref{prodasstar}, we have 
\begin{equation*}
s(\mathcal{T}_\lambda) = Rect(U^t*U)=U^t.U.
\end{equation*} It is clear that $U$ is uniquely determined by $S$. Hence, $g_{\lambda\mu}$ is the number of pairs $(T,U)$ in the set $\mathcal{T}(\mu^t,\mu,s(\mathcal{T}_\lambda))$ such that $T=U^t$. 
\end{proof}

Theorem \ref{gTtT} gives a way to compute the coefficients $g_{\lambda\mu}$. 
\begin{example}Let $\lambda=(5,2)$ and $\mu=(4,2,1)$. Since $\mu^t =(3,2,1,1)$, we can re-use the computation in Example \ref{ExampleLRrule}. The elements in the set $\mathcal{S}(\tilde{\lambda}/\mu^t,\mathcal{U}_\mu)$ with the corresponding elements in the set $\mathcal{T}(\mu^t,\mu,s(\mathcal{T}_\lambda))$ are (the elements in the subsets $\overline{\mathcal{T}(\mu^t,\mu,s(\mathcal{T}_\lambda))}$ are marked by coloring all boxes in green).
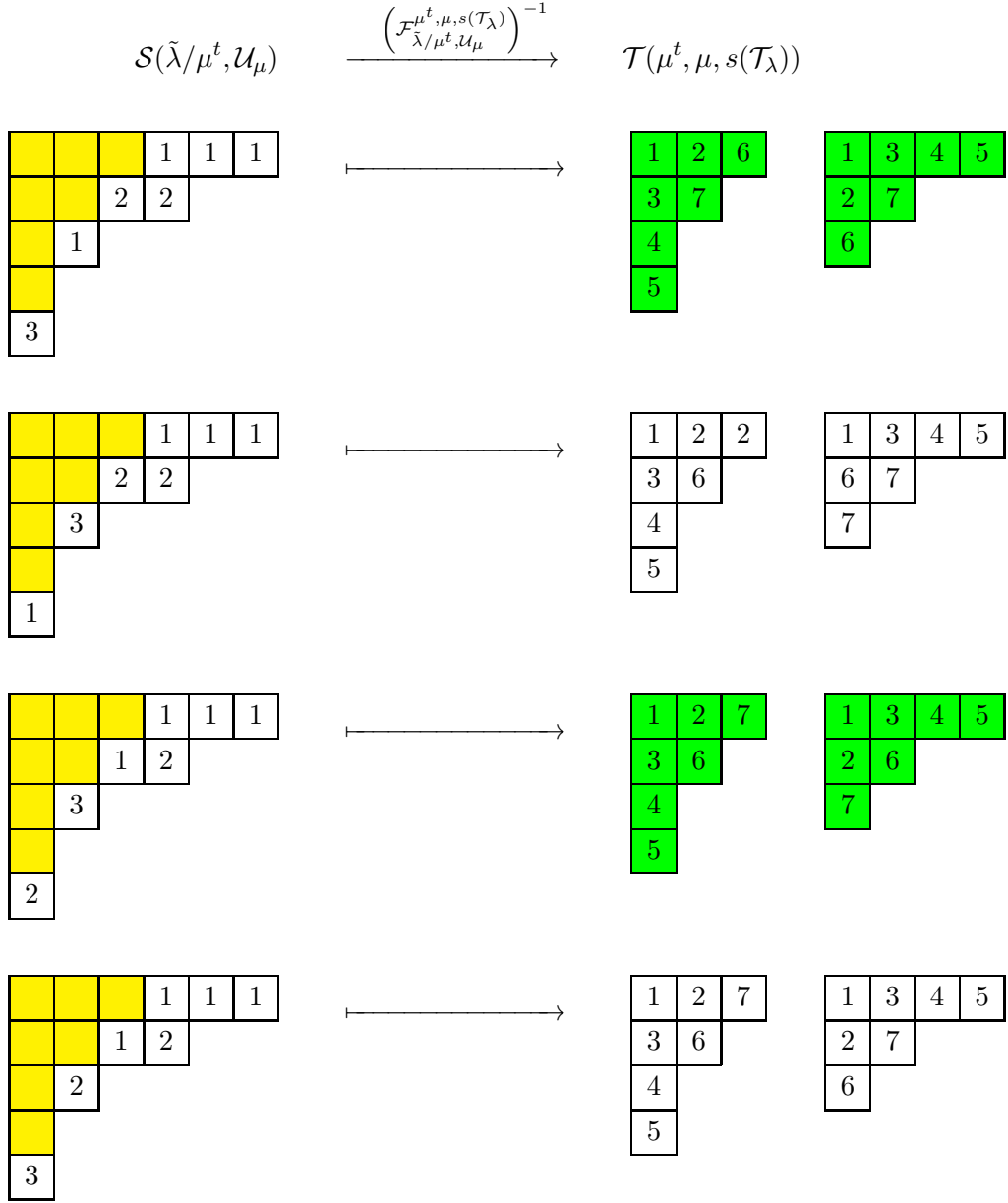
\begin{align*}
    \mathcal{S}(\tilde{\lambda}/\mu^t,\mathcal{U}_\mu) & \quad\quad \xrightarrow{\quad\left( \mathcal{F}^{\mu^t,\mu,s(\mathcal{T}_\lambda)}_{\tilde{\lambda}/\mu^t, \mathcal{U}_\mu} \right)^{-1}} &  \mathcal{T}(\mu^t,\mu,s(\mathcal{T}_\lambda))\quad\quad\quad\quad\quad\quad\quad\\
    \quad &\quad &\quad \\
    \begin{ytableau}
    *(yellow)&*(yellow)&*(yellow)&1&1&1\\
    *(yellow)&*(yellow)&2&2\\
    *(yellow)&1\\
    *(yellow)\\
    3
    \end{ytableau} & \quad\quad \xmapsto{\quad\quad\quad\quad\quad\quad\quad} &
    \begin{ytableau}
    *(green)1&*(green)2&*(green)6\\
    *(green)3&*(green)7\\
    *(green)4\\
    *(green)5
    \end{ytableau} \quad\quad
    \begin{ytableau}
    *(green)1&*(green)3&*(green)4&*(green)5\\
    *(green)2&*(green)7\\
    *(green)6
    \end{ytableau} \\
    \quad&\quad&\quad\\
    \begin{ytableau}
    *(yellow)&*(yellow)&*(yellow)&1&1&1\\
    *(yellow)&*(yellow)&2&2\\
    *(yellow)&3\\
    *(yellow)\\
    1
    \end{ytableau} & \quad\quad \xmapsto{\quad\quad\quad\quad\quad\quad\quad} &
    \begin{ytableau}
    1&2&2\\
    3&6\\
    4\\
    5
    \end{ytableau} \quad\quad
    \begin{ytableau}
   1&3&4&5\\
   6&7\\
   7
    \end{ytableau}\\
    \quad&\quad&\quad\\
    \begin{ytableau}
    *(yellow)&*(yellow)&*(yellow)&1&1&1\\
    *(yellow)&*(yellow)&1&2\\
    *(yellow)&3\\
    *(yellow)\\
    2
    \end{ytableau} & \quad\quad\xmapsto{\quad\quad\quad\quad\quad\quad\quad} &
    \begin{ytableau}
    *(green)1&*(green)2&*(green)7\\
    *(green)3&*(green)6\\
    *(green)4\\
    *(green)5
    \end{ytableau} \quad\quad
    \begin{ytableau}
   *(green)1&*(green)3&*(green)4&*(green)5\\
   *(green)2&*(green)6\\
   *(green)7
    \end{ytableau}\\
    \quad&\quad&\quad\\
    \begin{ytableau}
    *(yellow)&*(yellow)&*(yellow)&1&1&1\\
    *(yellow)&*(yellow)&1&2\\
    *(yellow)&2\\
    *(yellow)\\
    3
    \end{ytableau} &\quad\quad \xmapsto{\quad\quad\quad\quad\quad\quad\quad} &
    \begin{ytableau}
   1&2&7\\
   3&6\\
   4\\
   5
    \end{ytableau} \quad\quad
    \begin{ytableau}
  1&3&4&5\\
  2&7\\
  6
    \end{ytableau}
\end{align*}
Hence, $g_{\lambda\mu}=2$.
\end{example}

\subsection{Inequality \texorpdfstring{$g_{\lambda\mu} \leq c^{\tilde{\lambda}}_{\mu^t\mu}$}{TEXT}}
In this section, we prove a conjecture of Ressayre \cite{Ressayre19}, which originates from the main results of the paper \cite{BelkaleKumarRessayre}. \\

First, we need some notations. Let $G$ be a connected reductive complex algebraic group, $B$ a Borel subgroup, $H \subset B$ a maximal torus, $P \supset B$ a (standard) parabolic subgroup of $G$. Let $L$ be the Levi subgroup of $P$, $B_L$ the Borel subgroup of $L$, $L^{ss} = [L,L]$ the semisimple part of $L$. Let $X_L(H)^+$ be the set of $L$-dominant characters of $H$ via the highest weight. For each $\lambda \in X_L(H)^+$, let $V_L(\lambda)$ the irreducible representation of $L$ with highest weight $\lambda$, $\mathcal{L}_P(\lambda)$ the $P$-equivariant line bunble on $P/B_L$ associated to the principal $B_L$-bundle $P \rightarrow P/B_L$ via the one dimensional $B_L$-module $\lambda^{-1}$. Let $W$ be the Weyl group of $G$, $W_P$ the Weyl group of $P$, $W^P$ the set of minimal length coset representatives in $W/W_P$. For any $w \in W^P$, let $X_w$ be the corresponding Schubert variety, $[X_w] \in H^{2(\dim(G/P) - l(w))}(G/P, \mathbb{Z})$ the corresponding Poincaré dual class. Let $\astrosun_0$ be the deformed product in the singular cohomology $H^*(G/P, \mathbb{Z})$ (see \cite{BelkaleKumar}). Let $\rho$ and $\rho^L$ be the half sum of positive roots of $G$ and $L$ respectively. For any $w \in W^P$, set $\chi_w = \rho - 2 \rho^L + w^{-1}\rho$. \\

Below are the main results (Theorem 1.4 and Proposition 1.6) in the paper \cite{BelkaleKumarRessayre}.
\begin{theo}\label{Theorem 1.4 in BKR} Let $w_1, \dots , w_s \in W^P$ be such that 
\begin{equation*}
    [X_{w_1}] \astrosun_0 \dots \astrosun_0 [X_{w_s}] = [X_e] \in H^*(G/P).
\end{equation*}
Then for any positive integer $n$,
\begin{equation*}
    \dim([V_L(n \chi_{w_1})] \otimes \dots \otimes V_L(n\chi_{w_s})^{L^{ss}}) = 1.
\end{equation*}
\end{theo}

\begin{prop}\label{Proposition 1.6 in BKR}
Let $w_1, \dots, w_s \in W^P$ be such that 
\begin{equation*}
    [X_{w_1}] \astrosun_0 \dots \astrosun_0 [X_{w_s}] = d[X_e] \in H^*(G/P), \text{ for some }d \ne 0.
\end{equation*}
Then $\dim( H^0( (L/B_L)^s, (\mathcal{L}_P(\chi_{w_1}) \boxtimes \dots \boxtimes \mathcal{L}_P(\chi_{w_s}))_{|(L/B_L)^s}     )^{L^{ss}}) \ne 0$. 
\end{prop}

In this section, we prove a conjecture of Ressayre \cite{Ressayre19}: $g_{\lambda\mu} \leq c_{\mu^t \mu}^{\tilde{\lambda}}$. The conjecture was based on the facts below
\begin{equation}\label{11}
g_{\lambda\mu} = 1 \text{ implies }c_{\mu^t \mu}^{\tilde{\lambda}} =1,
\end{equation}
\begin{equation}\label{00}
g_{\lambda\mu} \ne 0 \text{ implies } c_{\mu^t \mu}^{\tilde{\lambda}} \ne 0.
\end{equation}
The conclusions (\ref{11}), (\ref{00}) are versions of Theorem \ref{Theorem 1.4 in BKR} and Proposition \ref{Proposition 1.6 in BKR} for the case $G=Sp(2n,\mathbb{C})$. Indeed, by \cite{Ressayre}, in the case $G=Sp(2n,\mathbb{C})$, $G/P$ is the Lagrangian Grassmannian $LG(n,\mathbb{C}^{2n})$, the corresponding Levi group is $GL(n)$, and $W^P$ is parametrized by strict paritions. Suppose that $w\in W^P$ corresponds to strict partition $\lambda$, then $\chi_w$ corresponds to the partition $\tilde{\lambda}$. Let $\lambda^\vee$ be the strict partition corresponding to the completion of $sY(\lambda)$ in $sY((n,n-1,\dots,1))$. Let $\lambda_1, \lambda_2,\lambda_3$ be the strict partitions corresponding to $w_1,w_2,w_3 \in W^P$, respectively in Theorem \ref{Theorem 1.4 in BKR}. Since the structure constants for the singular cohomology and the deformed cohomology $\astrosun_0$ in this case are the same, the theorem says that for any positive integer $k$, we have
\begin{equation}\label{easyTheorem1.4}
    f_{\lambda_1 \lambda_2 }^{\lambda_3^\vee} = 1 \text{ implies  }c_{\widetilde{k\lambda_1} \widetilde{k\lambda_2}}^{\widetilde{k\lambda_3^\vee}}=1. 
\end{equation}
In particular, when $\lambda_1=\lambda$, $\lambda_2 =\delta$, $\lambda_3^\vee =\mu +\delta$, the left-hand side of (\ref{easyTheorem1.4}) becomes 
$$
    1=f^{\mu+\delta}_{\lambda\delta}=g_{\lambda\mu}.
$$
With $k=1$, by Theorem \ref{LR}, the right-hand side of (\ref{easyTheorem1.4}) becomes
\begin{align*}
   1= c_{\tilde{\lambda}\tilde{\delta}}^{\widetilde{\mu+\delta}} &=\# \mathcal{S}(\widetilde{\mu+\delta}/\tilde{\delta},s(\mathcal{T}_\lambda))\\
    &=\# \mathcal{S}(\mu^t*\mu,s(\mathcal{T}_\lambda))\\
    &=\# \mathcal{T}(\mu^t,\mu,s(\mathcal{T}_\lambda))\\
    &=c^{\tilde{\lambda}}_{\mu^t\mu}.
\end{align*}
Hence, we get the conclusion (\ref{11}). Similarly, the conclusion (\ref{00}) follows from Proposition \ref{Proposition 1.6 in BKR}.
\begin{theo}\label{glessthanc}
Let $\lambda$ be a strict partition and let $\mu$ be a partition. Then $g_{\lambda\mu} \leq c^{\tilde{\lambda}}_{\mu^t\mu}$.
\end{theo}
\begin{proof}
By Theorem \ref{gTtT} and Theorem \ref{LR}, we have \begin{equation*}
g_{\lambda\mu} = \# \overline{\mathcal{T}(\mu^t,\mu,s(\mathcal{T}_\lambda))} \leq \# \mathcal{T}(\mu^t,\mu,s(\mathcal{T}_\lambda)) =  c^{\tilde{\lambda}}_{\mu^t\mu}.
\end{equation*}
\end{proof}
\subsection{Inequality \texorpdfstring{$g_{\lambda\mu}^2 \leq c^{\tilde{\lambda}}_{\mu^t \mu}$}{TEXT}}\label{g2<c}
In this section, we propose a stronger conjectural inequality than Theorem \ref{glessthanc}. We provide some examples to support this conjecture. Indeed, we formulate a conjecture on combinatorial models whose validity implies the first conjecture.   
\begin{conj}\label{conjectureinequal2}Let $\lambda$ be a strict partition and let $\mu$ be a partition. Then $g_{\lambda\mu}^2 \leq c^{\tilde{\lambda}}_{\mu^t \mu}$.
\end{conj}

To compute the decomposition of $P_\lambda$ in terms of $s_\mu$, we use SageMath (online version: \url{https://cocalc.com/}). For example, 
\begin{python}
Sym = SymmetricFunctions(FractionField(QQ['t']))
SP = Sym.hall_littlewood(t=-1).P();
s = Sym.schur();
s(SP([4,2]))
\end{python}
The result appears in computer is 
\begin{python}
s[2, 2, 1, 1] + s[2, 2, 2] + s[3, 1, 1, 1] + 2*s[3, 2, 1] + s[3, 3] + s[4, 1, 1] + s[4, 2].
\end{python}
It means $P_{(4,2)} = s_{(2,2,1,1)} + s_{(2,2,2)} + s_{(3,1,1)} + s_{(3,3)} + s_{(4,1,1)} + s_{(4,2)}.$\\

To compute the Littlewood-Richardson coefficients by computer, we use the code below. For example, 
\begin{python}
import sage.libs.lrcalc.lrcalc as lrcalc
A = [5,4,2,1]
B = [3,2,1]
C = [3,2,1]
lrcalc.lrcoef(A,B,C)
\end{python}
The result appears in computer is 
\begin{python}
4.
\end{python}
It means
$c_{\mu^t\mu}^{\tilde{\lambda}} = 4$, where $\lambda=(4,2)$, $\mu=\mu^t=(3,2,1)$.\\

We check the conjecture for all strict partitions $\lambda$ such that $|\lambda|\leq 11$. By Theorem \ref{glessthanc}, we just need to check the cases $g_{\lambda\mu} >1$. Here is the data of computations on computer.  \\
\newpage

\begin{longtable}{|c|c|c|c|c|}
  \hline
  $|\lambda|$ & $\lambda$ such that $g_{\lambda\mu}>1$ & $\mu$ such that $g_{\lambda\mu} >1$ &  $g_{\lambda\mu}$ & $c^{\tilde{\lambda}}_{\mu^t\mu}$\\
  \hline
11& (9,2)&(3, 2, 1, 1, 1, 1, 1, 1)&2&4 \\
11& (9,2)&(4, 2, 1, 1, 1, 1, 1)&2&4\\
11& (9,2)&(5, 2, 1, 1, 1, 1)&2&4\\
11& (9,2)&(6, 2, 1, 1, 1)&2&4\\
11& (9,2)&(7, 2, 1, 1)&2&4\\
11& (9,2)&(8, 2, 1)&2&4\\
  \hline
11& (8,3)&(3, 2, 2, 1, 1, 1, 1)&2&4 \\
11& (8,3)&(4, 2, 1, 1, 1, 1, 1)&2&4\\
11& (8,3)&(4, 2, 2, 1, 1, 1)&2&4\\
11& (8,3)&(4, 3, 1, 1, 1, 1)&2&4\\
11& (8,3)&(5, 2, 1, 1, 1, 1)&2&4\\
11& (8,3)&(5, 2, 2, 1, 1)&2&4\\
11& (8,3)&(5, 3, 1, 1, 1)&2&4\\
11& (8,3)&(6, 2, 1, 1, 1)&2&4\\
11& (8,3)&(6, 2, 2, 1)&2&4\\
11& (8,3)&(6, 3, 1, 1)&2&4\\
11& (8,3)&(7, 2, 1, 1)&2&4\\
11& (8,3)&(7, 3, 1)&2&4\\
\hline
11&(7,4)&(3, 2, 2, 2, 1, 1)&2&4\\
11&(7,4)&(4, 2, 2, 1, 1, 1)&2&4\\
11&(7,4)&(4, 2, 2, 2, 1)&2&4\\
11&(7,4)&(4, 3, 2, 1, 1)&2&4\\
11&(7,4)&(5, 2, 1, 1, 1, 1)&2&4\\
11&(7,4)&(5, 2, 2, 1, 1)&2&4\\
11&(7,4)&(5, 3, 1, 1, 1)&2&4\\
11&(7,4)&(5, 3, 2, 1)&2&4\\
11&(7,4)&(5, 4, 1, 1)&2&4\\
11&(7,4)&(6, 2, 1, 1, 1)&2&4\\
11&(7,4)&(6, 3, 1, 1)&2&4\\
11&(7,4)&(6, 4, 1)&2&4\\
\hline
11&(7,3,1)&(3, 3, 2, 1, 1, 1)&2&6\\
11&(7,3,1)&(4, 2, 2, 1, 1, 1)&2&5\\
11&(7,3,1)&(4, 3, 1, 1, 1, 1)&2&5\\
11&(7,3,1)&(4, 3, 2, 1, 1)&3&13\\
11&(7,3,1)&(5, 2, 2, 1, 1)&2&5\\
11&(7,3,1)&(5, 3, 1, 1, 1)&2&5\\
11&(7,3,1)&(5, 3, 2, 1)&3&13\\
11&(7,3,1)&(6, 2, 2, 1)&2&5\\
11&(7,3,1)&(6, 3, 1, 1)&2&5\\
11&(7,3,1)&(6, 3, 2)&2&6\\
\hline
11&(6,4,1)&(3, 3, 2, 2, 1)&2&6\\
11&(6,4,1)&(4, 2, 2, 2, 1)&2&5\\
11&(6,4,1)&(4, 3, 2, 1, 1)&3&14\\
11&(6,4,1)&(4, 3, 2, 2)&2&7\\
11&(6,4,1)&(4, 3, 3, 1)&2&4\\
11&(6,4,1)&(4, 4, 2, 1)&2&7\\
11&(6,4,1)&(5, 2, 2, 1, 1)&2&5\\
11&(6,4,1)&(5, 3, 1, 1, 1)&2&5\\
11&(6,4,1)&(5, 3, 2, 1)&3&14\\
11&(6,4,1)&(5, 4, 1, 1)&2&5\\
11&(6,4,1)&(5, 4, 2)&2&6\\
\hline
11& (6,3,2)&(4, 3, 2, 1, 1) &3&10\\
11& (6,3,2)&(4, 3, 2, 2)&2&4\\
11& (6,3,2)&(4, 3, 3, 1)&2&5\\
11& (6,3,2)&(4, 4, 2, 1)&2&4\\
11& (6,3,2)&(5, 3, 2, 1)&3&10\\
\hline
11&(5,4,2)&(4, 3, 2, 1, 1)&2&4\\
11&(5,4,2)&(4, 3, 2, 2)&2&5\\
11&(5,4,2)&(4, 3, 3, 1)&2&5\\
11&(5,4,2)&(4, 4, 2, 1)&2&5\\
11&(5,4,2)&(5, 3, 2, 1)&2&4\\
\hline
10&(8,2)&(3, 2, 1, 1, 1, 1, 1)&2&4 \\
10&(8,2)&(4, 2, 1, 1, 1, 1)&2&4\\
10&(8,2)&(5, 2, 1, 1, 1)&2&4\\
10&(8,2)&(6, 2, 1, 1)&2&4\\
10&(8,2)&(7, 2, 1)&2&4\\
\hline
10&(7,3)& (3, 2, 2, 1, 1, 1)&2&4\\
10&(7,3)&(4, 2, 1, 1, 1, 1)&2&4\\
10&(7,3)&(4, 2, 2, 1, 1)&2&4\\
10&(7,3)&(4, 3, 1, 1, 1)&2&4\\
10&(7,3)&(5, 2, 1, 1, 1)&2&4\\
10&(7,3)&(5, 2, 2, 1)&2&4\\
10&(7,3)&(5, 3, 1, 1)&2&4\\
10&(7,3)&(6, 2, 1, 1)&2&4\\
10&(7,3)&(6, 3, 1) &2&4\\
\hline
10&(6,4)&(3, 2, 2, 2, 1)&2&4 \\
10&(6,4)&(4, 2, 2, 1, 1)&2&4 \\
10&(6,4)&(4, 3, 2, 1)&2&4 \\
10&(6,4)&(5, 2, 1, 1, 1)&2&4 \\
10&(6,4)&(5, 3, 1, 1)&2&4 \\
10&(6,4)&(5, 4, 1)&2&4 \\
\hline
10&(6,3,1)&  (3, 3, 2, 1, 1)&2&6\\
10&(6,3,1)&(4, 2, 2, 1, 1)&2&5 \\
10&(6,3,1)&(4, 3, 1, 1, 1)&2&5\\
10&(6,3,1)&(4, 3, 2, 1)&3&13\\
10&(6,3,1)&(5, 2, 2, 1)&2&5\\
10&(6,3,1)&(5, 3, 1, 1)&2&5\\
10&(6,3,1)&(5, 3, 2)&2&6\\
\hline
10&(5,4,1)&(4, 3, 2, 1)&2&7\\
\hline
10&(5,3,2)&(4, 3, 2, 1)&3&9\\ 
\hline
9&(7,2)&(3, 2, 1, 1, 1, 1)&2&4\\
9&(7,2)&(4, 2, 1, 1, 1)&2&4 \\
9&(7,2)&(5, 2, 1, 1)&2&4\\
9&(7,2)&(6, 2, 1)&2&4\\ 
\hline
9&(6,3)&(3, 2, 2, 1, 1)&2&4\\
9&(6,3)&(4, 2, 1, 1, 1)&2&4\\
9&(6,3)&(4, 2, 2, 1)&2&4\\
9&(6,3)&(4, 3, 1, 1)&2&4\\
9&(6,3)&(5, 2, 1, 1)&2&4\\
9&(6,3)&(5, 3, 1)&2&4\\
\hline
9&(5,3,1)&(3, 3, 2, 1)&2&6\\
9&(5,3,1)&(4, 2, 2, 1)&2&5\\
9&(5,3,1)&(4, 3, 1, 1)&2&5\\
9&(5,3,1)&(4, 3, 2)&2&6 \\
\hline
8&(6,2)& (3, 2, 1, 1, 1)&2&4 \\
8&(6,2)&(4, 2, 1, 1)&2&4 \\
8&(6,2)&(5, 2, 1)&2&4\\
\hline
8&(5,3)& (3, 2, 2, 1)&2&4\\
8&(5,3)&(4, 2, 1, 1)&2&4\\
8&(5,3)&(4, 3, 1)&2&4\\
\hline
7&(5,2)& (3, 2, 1, 1)&2&4\\
7&(5,2)&(4, 2, 1)&2&4\\
\hline
6&(4,2)&(3,2,1)&2&4\\
\hline
$< 6$& $\emptyset$ & $\emptyset$ & &\\
\hline
\end{longtable}
\begin{conj}\label{bij} We have
\begin{itemize}
\item[1.]
The restriction of the map $\mathcal{S}^{\mu^t,\mu,\tilde{\lambda}}_{s(\mathcal{T}_\lambda), \mathcal{U}_\mu, \mathcal{U}_{\mu^t},s(\mathcal{T}_\lambda)}$ on the set $\overline{\mathcal{T}(\mu^t,\mu,s(\mathcal{T}_\lambda))}$ is a bijection onto the set $\overline{\mathcal{T}(\mu,\mu^t,s(\mathcal{T}_\lambda))}$.
\item[2.] The elements of the set $\overline{\mathcal{T}(\mu^t,\mu,s(\mathcal{T}_\lambda))}$ have the form $(U^t_\alpha,U_\alpha)$, with index $\alpha$. Let $(V_\alpha,V_\alpha^t)$ be the image of $(U_\alpha^t,U_\alpha)$ through the bijection $\mathcal{S}^{\mu^t,\mu,\tilde{\lambda}}_{s(\mathcal{T}_\lambda), \mathcal{U}_\mu, \mathcal{U}_{\mu^t},s(\mathcal{T}_\lambda)}$. Let $(U_\alpha^t,U_\alpha)$ and $(U_\beta^t,U_\beta)$ be elements of the set $\overline{\mathcal{T}(\mu^t,\mu,s(\mathcal{T}_\lambda))}$. If $(U_\alpha^t,U_\beta)$ is not in the set $\mathcal{T}(\mu^t,\mu,s(\mathcal{T}_\lambda))$, then $(V_\alpha,V_\beta^t)$ is in the set $\mathcal{T}(\mu,\mu^t,s(\mathcal{T}_\lambda))$.
\end{itemize}
\end{conj}

\begin{rema}\label{coro_of_conj} Thanks to Theorem \ref{gTtT}, the validity of Conjecture \ref{bij} 1. implies the equality $g_{\lambda\mu}=g_{\lambda\mu^t}$, which was proved in Proposition \ref{gmumut}.  
\end{rema}

\begin{prop}\label{g2lessthanc} Suppose that Conjecture \ref{bij} holds. Then we have $g_{\lambda\mu}^2 \leq c^{\tilde{\lambda}}_{\mu^t\mu}$.
\end{prop}
\begin{proof}
We suppose that $(U_\alpha^t,U_\alpha)$ and $(U_\beta^t,U_\beta)$ are elements in $\mathcal{T}(\mu^t,\mu,s(\mathcal{T}_\lambda))$. We construct an element $\widetilde{(U_\alpha^t,U_\beta)}$ in the set $\mathcal{T}(\mu^t,\mu,s(\mathcal{T}_\lambda))$ as follows: 
\begin{itemize}
\item[1.] If $(U_\alpha^t,U_\beta)$ belongs to $\mathcal{T}(\mu^t,\mu,s(\mathcal{T}_\lambda))$, then we set $\widetilde{(U_\alpha^t,U_\beta)} = (U_\alpha^t,U_\beta)$.
\item[2.] If $(U_\alpha^t,U_\beta)$ does not belong to $\mathcal{T}(\mu^t,\mu,s(\mathcal{T}_\lambda))$, then by Conjecture \ref{bij}, $(V_\alpha,V_\beta^t)$ belongs to $\mathcal{T}(\mu,\mu^t,s(\mathcal{T}_\lambda))$.
Set $\widetilde{(U_\alpha^t,U_\beta)}$ is the image of $(V_\alpha,V_\beta^t)$ through the bijection $\left(\mathcal{S}^{\mu^t,\mu,\tilde{\lambda}}_{s(\mathcal{T}_\lambda), \mathcal{U}_\mu, \mathcal{U}_{\mu^t},s(\mathcal{T}_\lambda)}\right)^{-1}$. 
\end{itemize}
The set of all pairs $\widetilde{(U_\alpha ^t, U_\beta)}$ we have constructed is a subset of $\mathcal{T}(\mu^t,\mu,s(\mathcal{T}_\lambda))$. Since its cardinal is $g_{\lambda\mu}^2$, we have $g_{\lambda\mu}^2 \leq c^{\tilde{\lambda}}_{\mu^t\mu}$.
\end{proof}
We can see the conjecture through the following example. 
\begin{example}Let $\lambda=(5,2)$ and $\mu=(4,2,1)$. The correspondence between elements in $\mathcal{T}(\mu^t,\mu,s(\mathcal{T}_\lambda))$ and elements in $\mathcal{T}(\mu,\mu^t,s(\mathcal{T}_\lambda))$ is showed below (the elements in the subsets $\overline{\mathcal{T}(\mu^t,\mu,s(\mathcal{T}_\lambda))}$ and $\overline{\mathcal{T}(\mu,\mu^t,s(\mathcal{T}_\lambda))}$ are marked by coloring all boxes in green). 
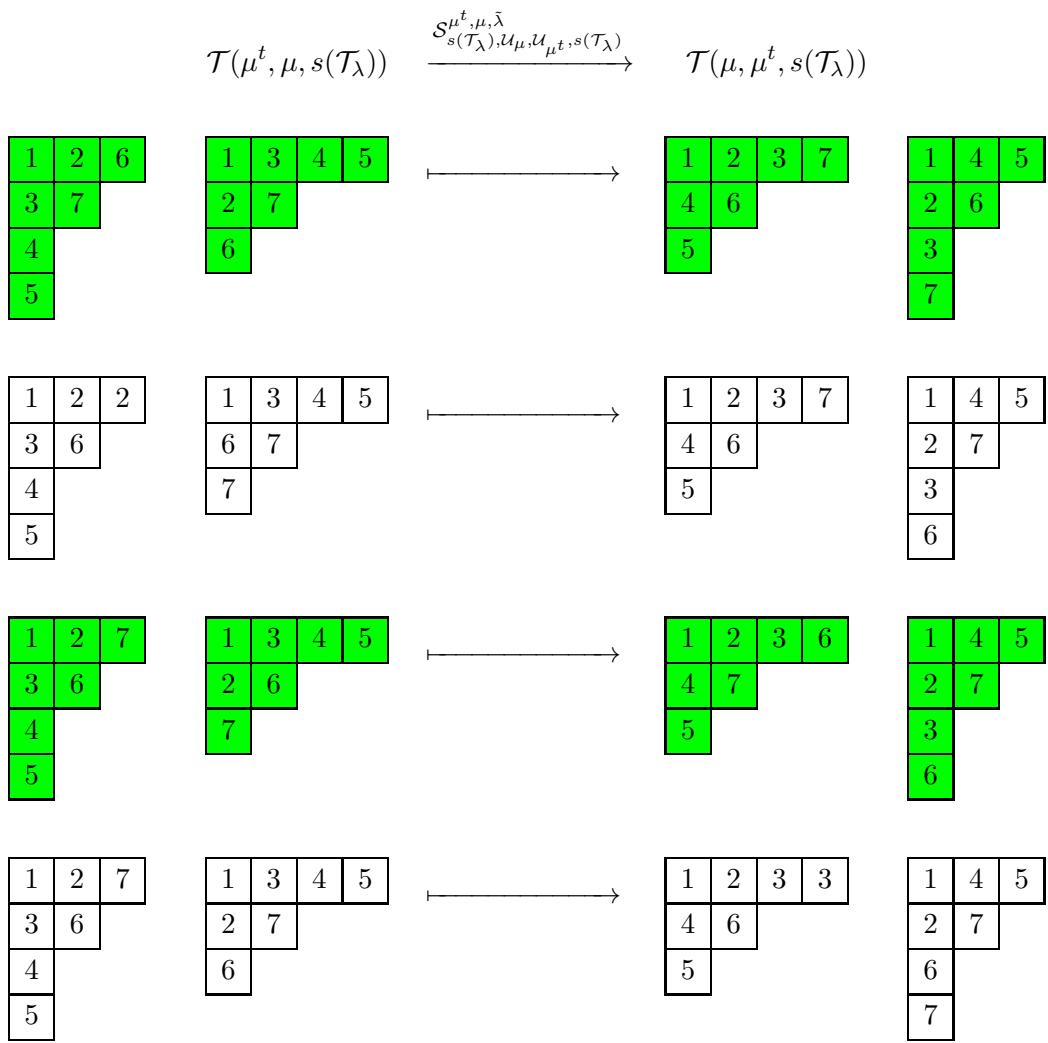
\begin{align*}
    \mathcal{T}(\mu^t,\mu,s(\mathcal{T}_\lambda)) & \quad \xrightarrow{\mathcal{S}^{\mu^t,\mu,\tilde{\lambda}}_{s(\mathcal{T}_\lambda),\mathcal{U}_\mu,\mathcal{U}_{\mu^t},s(\mathcal{T}_\lambda)}} &  \mathcal{T}(\mu,\mu^t,s(\mathcal{T}_\lambda))\quad\quad\quad\quad\quad\quad&\quad\\
    \quad &\quad &\quad&\quad \\
    \begin{ytableau}
    *(green)1&*(green)2&*(green)6 \\
    *(green)3&*(green)7\\
    *(green)4\\
    *(green)5
    \end{ytableau}\quad\quad
    \begin{ytableau}
    *(green)1&*(green)3&*(green)4&*(green)5\\
    *(green)2&*(green)7\\
    *(green)6 
    \end{ytableau}
&\quad\xmapsto{\quad\quad\quad\quad\quad\quad}
& 
\begin{ytableau}
    *(green)1&*(green)2&*(green)3&*(green)7 \\
    *(green)4&*(green)6\\
    *(green)5
    \end{ytableau}\quad\quad
    \begin{ytableau}
    *(green)1&*(green)4&*(green)5\\
    *(green)2&*(green)6\\
    *(green)3\\
    *(green)7
    \end{ytableau}
& \\
\quad &\quad &\quad& \\
\begin{ytableau}
1&2&2\\
3&6\\
4\\
5
\end{ytableau}
\quad\quad
\begin{ytableau}
1&3&4&5\\
6&7\\
7
\end{ytableau} 
&\quad\xmapsto{\quad\quad\quad\quad\quad\quad}
&
\begin{ytableau}
1&2&3&7\\
4&6\\
5
\end{ytableau}
\quad\quad
\begin{ytableau}
1&4&5\\
2&7\\
3\\
6
\end{ytableau}
&
\\
\quad &\quad &\quad& \\
    \begin{ytableau}
    *(green)1&*(green)2&*(green)7 \\
    *(green)3&*(green)6\\
    *(green)4\\
    *(green)5
    \end{ytableau}\quad\quad
    \begin{ytableau}
    *(green)1&*(green)3&*(green)4&*(green)5\\
    *(green)2&*(green)6\\
    *(green)7
    \end{ytableau}
&\quad\xmapsto{\quad\quad\quad\quad\quad\quad}
& 
\begin{ytableau}
    *(green)1&*(green)2&*(green)3&*(green)6 \\
    *(green)4&*(green)7\\
    *(green)5
    \end{ytableau}\quad\quad
    \begin{ytableau}
    *(green)1&*(green)4&*(green)5\\
    *(green)2&*(green)7\\
    *(green)3\\
    *(green)6
    \end{ytableau} \\
\quad &\quad &\quad& \\
\begin{ytableau}
1&2&7\\
3&6\\
4\\
5
\end{ytableau}
\quad\quad
\begin{ytableau}
1&3&4&5\\
2&7\\
6
\end{ytableau} 
&\quad\xmapsto{\quad\quad\quad\quad\quad\quad}
&
\begin{ytableau}
1&2&3&3\\
4&6\\
5
\end{ytableau}
\quad\quad
\begin{ytableau}
1&4&5\\
2&7\\
6\\
7
\end{ytableau}
&
\end{align*}
\end{example}

\section*{Acknowledgments}
\addcontentsline{toc}{section}{Acknowledgements}
The author would like to express his sincere gratitude to his supervisors Prof. Nicolas Ressayre and Prof. Kenji Iohara for suggesting the subject and for many useful discussion, inspiring ideas during the work. He is also grateful to their corrections and their teaching of how to write better, understandable and clear. He would also like to thank Prof. Cédric Lecouvey and Prof. Shrawan Kumar - the reporters for his thesis, for reading it carefully, pointing out several important errors that improved the text. The author is grateful to the referees for their comments to improve the manuscript.
\addcontentsline{toc}{section}{References}
\bibliography{references}{}
\bibliographystyle{alpha}
\noindent Universite Lyon, University Claude Bernard Lyon 1, CNRS UMR 520, Institut Camille Jordan, 43 Boulevard du 11 Novembre 1918, F-69622 Villeurbanne Cedex, France\\
E-mail: \href{mailto:khanh.mathematic@gmail.com}{khanh.mathematic@gmail.com}
\appendix \label{Appendix}
\newpage
\begin{sidewaystable}
\begin{tiny}
\begin{tikzpicture}[nodes={}, ->]
\tikzstyle{level 1}=[sibling distance=85mm]
\tikzstyle{level 2}=[sibling distance=45mm]
\tikzstyle{level 3}=[sibling distance=25mm]
\tikzstyle{level 4}=[sibling distance=30mm]
\node{$\begin{ytableau}  1  \end{ytableau}$}
    child
    {
        node{$\begin{ytableau}  1&2  \end{ytableau}$}
            child
            {
                node{$\begin{ytableau}  1&2\\3  \end{ytableau}$}
                    child{node{$\begin{ytableau}  1&2\\3\\4  \end{ytableau}$}
                        child{node{$\begin{ytableau}  1&2\\3&5\\4  \end{ytableau}$}}
                        child{node{$\begin{ytableau}  1&2&5\\3\\4  \end{ytableau}$}}
                    }
            }
            child
            {
                node{$\begin{ytableau}  1&2\\3'  \end{ytableau}$}
                    child{node{$\begin{ytableau}  1&2 \\3'&4 \end{ytableau}$}
                        child{node{$\begin{ytableau}  1&2&5\\3'&4  \end{ytableau}$}}
                    }
            }
    }
    child
    {
        node{$\begin{ytableau}  1&2'  \end{ytableau}$}
            child
            {
                node{$\begin{ytableau}  1&2'\\3'  \end{ytableau}$}
                    child{node{$\begin{ytableau}  1&2'\\3'&4   \end{ytableau}$}
                        child{node{$\begin{ytableau}  1&2'&5\\3'&4   \end{ytableau}$}}
                    }
            }            
            child
            {
                node{$\begin{ytableau}  1  &2'&3 \end{ytableau}$}
                    child{node{$\begin{ytableau}  1&2'&3\\4   \end{ytableau}$}
                        child{node{$\begin{ytableau}  1&2'&3&5\\4   \end{ytableau}$}}
                    }
            }
            child
            {
                node{$\begin{ytableau}  1 &2' &3'  \end{ytableau}$}
                    child{node{$\begin{ytableau}  1 &2' &3' &4 \end{ytableau}$}
                        child{node{$\begin{ytableau}  1 &2' &3' &4 &5 \end{ytableau}$}}
                    }
                    child{node{$\begin{ytableau}  1 &2' &3'&4'  \end{ytableau}$}
                        child{node{$\begin{ytableau}  1 &2' &3'&4'&5  \end{ytableau}$}}
                    }
                    child{node{$\begin{ytableau}  1 &2' &3' \\4 \end{ytableau}$}
                        child{node{$\begin{ytableau}  1 &2' &3'&5 \\4 \end{ytableau}$}}
                    }
            } 
    };
\end{tikzpicture}
\end{tiny}
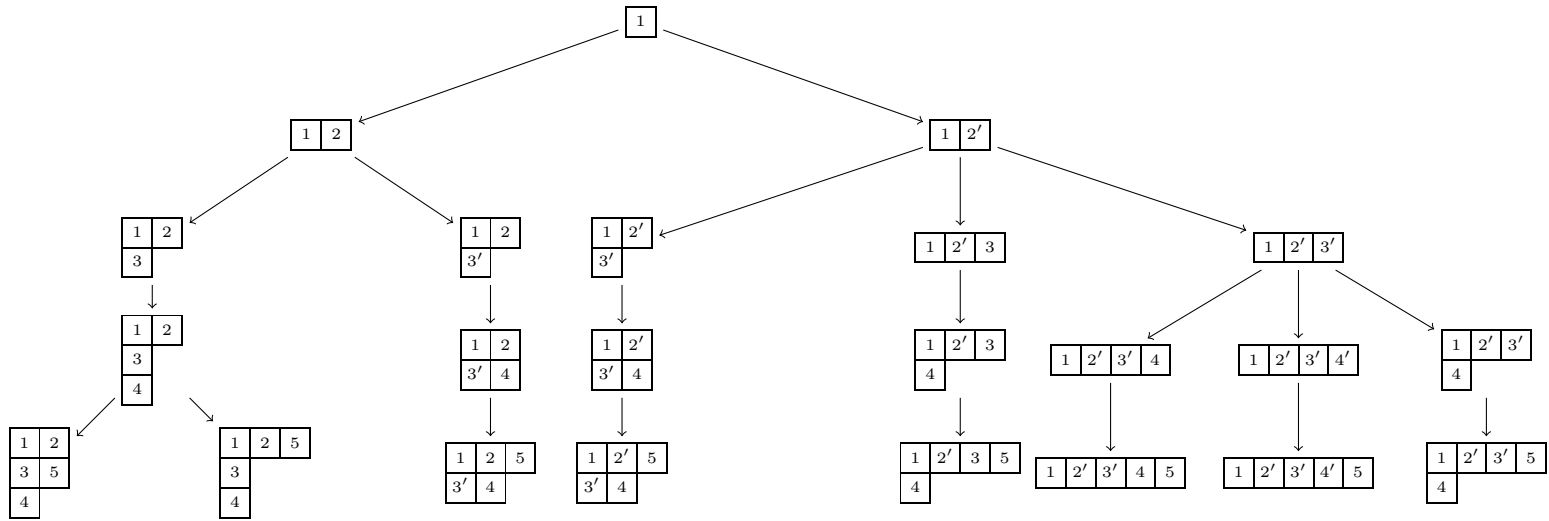
\\ Appendix. The construction of the set $\widetilde{\mathcal{O}}(\nu/\mu)$ in the Example \ref{example_f}.
\end{sidewaystable}
\end{document}